\documentclass[12pt,a4paper,reqno]{amsart} %Fr o m detta
\usepackage{amssymb}
\usepackage{latexsym}
\usepackage{amsmath}
\usepackage{mathrsfs}
\usepackage{euscript}
\usepackage[X2,T1]{fontenc}
\usepackage{cite}
\usepackage{calc}                   %Detta beh?vs f?r att
\newcommand{\scal}[2]{\langle #1,#2\rangle}
\newcommand{\rr}[1]{\mathbf R^{#1}}

\newcommand{\zz}[1]{\mathbf Z^{#1}}

\newcommand{\nn}[1]{\mathbf N^{#1}}
\newcommand{\nm}[2]{\Vert #1\Vert _{#2}}
\newcommand{\Nm}[2]{\left \Vert #1\right \Vert _{#2}}
\newcommand{\nmm}[1]{\Vert #1\Vert }

\newcommand{\sets}[2]{\{ \, #1\, ;\, #2\, \} }

\newcommand{\ep}{\varepsilon}
\newcommand{\fy}{\varphi}
\newcommand{\cdo}{\, \cdot \, }

\newcommand{\eabs}[1]{\langle #1\rangle}
\newcommand{\vrum}{\vspace{0.1cm}}

\newcommand{\maclE}{\mathcal E}

\newcommand{\maclL}{\mathcal L}
\newcommand{\maclS}{\mathcal S}
\newcommand{\mascB}{\mathscr B}

\newcommand{\mascF}{\mathscr F}

\newcommand{\mascP}{\mathscr P}

\newcommand{\mabfk}{\boldsymbol k}
\newcommand{\mabfp}{{\boldsymbol p}}
\newcommand{\mabfq}{\boldsymbol q}
\newcommand{\mabfr}{\boldsymbol r}
\newcommand{\mabfz}{\boldsymbol z}

\newcommand{\Co}{\mathsf{Co}}

\newcommand{\splM}{\EuScript M}

\newcommand{\splW}{\EuScript W}

\renewcommand{\projlim}[1]{\underset{#1}{\operatorname{proj \, lim\, }}}

\newcommand{\sfW}{\mathsf{W}}

\numberwithin{equation}{section}          %Detta g?r att man f?r
\newtheorem{thm}{Theorem}
\numberwithin{thm}{section}
\newtheorem*{tom}{\rubrik}
\newcommand{\rubrik}{}
\newtheorem{prop}[thm]{Proposition}

\newtheorem{lemma}[thm]{Lemma}

\theoremstyle{definition}

\newtheorem{defn}[thm]{Definition}

\theoremstyle{remark}
\newtheorem{rem}[thm]{Remark}              %T o m hit  bara allm
\title{Wiener estimates on modulation spaces}

\author{Joachim Toft}

\address{Department of Mathematics,
Linn{\ae}us University, V{\"a}xj{\"o}, Sweden}

\email{joachim.toft@lnu.se}



\begin{document}

\par

\begin{abstract}
We characterise modulation spaces by suitable Wiener
estimates on the short-time Fourier transforms of
the involved functions and distributions. We use the results to
refine some formulae on periodic distributions with
Lebesgue estimates on their coefficients.
\end{abstract}

\maketitle

%%%%%%%%%%%%%%%%%%%%%%
\section{Introduction}\label{sec0}
%%%%%%%%%%%%%%%%%%%%%%

In the paper we characterise
Gelfand-Shilov spaces of functions and distributions,
modulation spaces and Gevrey classes in background of
various kinds of Wiener estimates. We apply the results
to deduce some refined formulae on periodic %and quasi-periodic
functions and distributions, given in \cite{ToNa}.

\par

Essential motivations arised in
\cite{ToNa} on characterizations of certain spaces of
periodic functions and distributions.
In fact, it follows from \cite{ToNa} that if $q\in (0,\infty ]$
and $f$ is a $2\pi$-periodic Gelfand-Shilov distribution on $\rr d$ with
Fourier coefficients $c(f,\alpha )$, $\alpha \in \zz d$, then
\begin{alignat}{4}
&\{ c(f,\alpha )\} _{\alpha \in \zz d} \in \ell ^q &
\quad & &\quad 
&\Leftrightarrow & \quad f &\in M^{\infty ,q} .
\label{Eq:IntrEquiv3}
\end{alignat}
Here $M^{\infty ,q}$ is the (unweighted) modulation
spaces with Lebesgue parameters $\infty$ and $q$. (See Section
\ref{sec1} or \cite{ToNa} for notations.)
We note that a proof of \eqref{Eq:IntrEquiv3} in the case
$q\in [1,\infty]$ can be found in e.{\,}g. \cite{RuSuToTo},
and with some extensions in \cite{Re}.

\par

An alternative formulation of \eqref{Eq:IntrEquiv3} is
\begin{align}
\{ c(f,\alpha )\} _{\alpha \in \zz d} \in \ell ^q 
\qquad 
&\Leftrightarrow \qquad \xi \mapsto \nm {V_\phi f(\cdo ,\xi )}{L^\infty (\rr d)} \in L^q .
\tag*{(\ref{Eq:IntrEquiv3})$'$}
\intertext{By observing that periodicity of $f$ induce the same periodicity
for $x\mapsto |V_\phi f(x,\xi )|$, it follows that
%if $Q=[0,2\pi ]^d$, then
\eqref{Eq:IntrEquiv3}$'$ is the same as}
\{ c(f,\alpha )\} _{\alpha \in \zz d} \in \ell ^q 
\qquad 
&\Leftrightarrow \qquad \xi \mapsto \nm {V_\phi f(\cdo ,\xi )}{L^\infty ([0,2\pi ]^d)} \in L^q .
\tag*{(\ref{Eq:IntrEquiv3})$''$}
\intertext{
\indent
In Section \ref{sec2} we show that the latter equivalence hold true
with $L^r([0,2\pi]^d)$ norm in place of $L^\infty ([0,2\pi]^d)$ norm for every
$r\in (0,\infty ]$. That is, we improve \eqref{Eq:IntrEquiv3}$''$ into}
\{ c(f,\alpha )\} _{\alpha \in \zz d} \in \ell ^q 
\qquad 
&\Leftrightarrow \qquad \xi \mapsto \nm {V_\phi f(\cdo ,\xi )}{L^r ([0,2\pi ]^d)} \in L^q .
\tag*{(\ref{Eq:IntrEquiv3})$'''$}
\end{align}
In particular, if $q<\infty$ and choosing $q=r$, then we obtain
\begin{equation}\label{Eq:IntrEquiv3A}
\sum _{\alpha \in \zz d}|c(f,\alpha ) |^q <\infty
\qquad 
\Leftrightarrow \qquad 
\iint _{[0,2\pi ]^d\times \rr d} |V_\phi f(x,\xi )|^q\, dxd\xi <\infty .
\end{equation}
More generally, we deduce weighted versions of these
identities. Since our weights include general moderate
weights which are allowed possess exponential types
growth and decays, we formulate our results in
the framework of Gelfand-Shilov spaces of functions
and distributions.

\par

The improved equivalence \eqref{Eq:IntrEquiv3}$'''$ can
in the case $q,r\in [1,\infty]$ be obtained from \eqref{Eq:IntrEquiv3}$''$
by a suitable combination of H{\"o}lder's and Young's inequalities
and the inequality
\begin{align}
F(X) &\lesssim \int _\Omega 
\Phi (X-Y) F(Y )\, dY ,
\quad X\in \Omega =[0,2\pi]^d\times \rr d,
\label{Eq:ContConvPer}
\intertext{where}
\Phi (x,\xi ) &= \sum _{k\in \zz d}|V_\phi \phi (x-2\pi k,\xi )|
\quad \text{and}\quad F(X) = |V_\phi f (X)|,
\notag
\end{align}
which follows from Lemma 1.3.3 in \cite{Gc2} for $2\pi$-periodic
distributions $f$. It follows that this case can be handled by
straight-forward modifications of the methods that are used
when establishing basic results for classical modulation
spaces in \cite{F1} and in Chapter 11 in \cite{Gc2}.

\par

In our situation, the parameters $q$ and $r$ are, more generally, allowed
to belong to the full
interval $(0,\infty ]$ instead of $[1,\infty ]$. The classical approaches in
\cite{F1,FG1,FG2,Gc2}
are then insufficient because they require convex structures in the topology of
the involved vector spaces. This convexity is absent when $q<1$ or $r<1$.
%
%
%in some situations, convex structures in the topology of
%the involved vector spaces fail.

\par

We manage our more general situation by %more far-reaching
using techniques based on ideas in \cite{GaSa,Rau1,Rau2,Toft15}
and which can handle Lebesgue and Wiener spaces which are
quasi-Banach spaces but may fail to be Banach spaces. Especially
we shall follow a main idea in \cite{GaSa,Toft15}
and replace the usual convolution, used in
\cite{F1,FG1,FG2,Gc2}, by a semi-continuous version which is less
sensitive when convexity is lacking in the topological structures.
For the semi-continuous convolution we deduce in Section \ref{sec2}
the needed Lebesgue and Wiener estimates.
%for the semi-continuous convolution.
%
%it is shown that the key arguments when using the usual convolution in
%\cite{F1,FG1,FG2,Gc2} can be replaced by a semi-continuous version
%which is less sensitive of lack of convexity.
%
%\par
%
%Essential parts of our investigations are given in Section \ref{sec2}, were
%we deduce the needed key estimates of the semi-continuous convolution
%on Lebesgue or Wiener spaces.
In the end we achieve in Section \ref{sec2}
various types of characterizations of modulation spaces in terms
of Wiener norm estimates on the short-time Fourier transforms of
the functions and (ultra-)distributions under considerations. For example,
as special case of Propositions \ref{Prop:WienerEquiv}$'$ after Proposition
\ref{Prop:SemiContConvEstNonPer}, we have for $p,q,r\in (0,\infty ]$ that
\begin{equation}\label{Eq:ModSpaceIntrChar}
\nm f{M^{p,q}} \asymp \nm a{\ell ^{p,q}}
\quad \text{when}\quad
a(j) = \nm {V_\phi f}{L^r(j+[0,1]^{2d})}.
\end{equation}
Similar facts hold true for those Wiener amalgam spaces which are Fourier images
of modulation spaces of the form $M^{p,q}$. In particular our results can be
used to deduce certain invarians properties concerning the choice of
local component in the Wiener amalgam quasi-norm. (See also Proposition
\ref{Prop:SpecCaseWienerRel1}.) Here we remark that for Wiener amalgam
spaces which at the same time are Banach spaces, the approaches are often
less complicated and there are several
examples on other Banach spaces (e.{\,}g. suitable modulation spaces)
to furnish the local component in the Wiener amalgam norms.
(See e.{\,}g. \cite{FeLu,FeZi} and the references therein.)

\par

We also present some applications on periodic elements which gives
\eqref{Eq:IntrEquiv3}$'''$ and \eqref{Eq:IntrEquiv3A} as special cases. (See
Propositions \ref{Prop:PeriodicMod} and \ref{Prop:PerMod}$'$.)

\par

The Wiener spaces under considerations
can also be described in terms of coorbit spaces, whose general
theory was founded
by Feichtinger and Gr{\"o}chenig in \cite{FG1,FG2} and further developed
in different ways, e.{\,}g. by Rauhut in \cite{Rau1,Rau2}. Since our
investigations in Section \ref{sec2} concern quasi-Banach spaces
which may fail to be Banach spaces, our investigations are especially
linked to Rauhut's analysis in \cite{Rau1,Rau2}. 
%
%For short-time Fourier
%transforms $V_\phi f$
%of $f$, it follows that in some situations, our Wiener norm estimates on
%$V_\phi f$ are weighted versions of Wiener amalgam norm estimates of
%$f$ or of $\widehat f$, with local components in Fourier Lebesgue spaces
%and global components in Lebesgue spaces.
%
In this context, a part of our analysis on modulation spaces can be formulated
as coorbit norm estimates of short-time Fourier transforms with local component
in $L^r$-spaces with $r\in (0,\infty ]$ and global component in other Lebesgue 
spaces. Proposition \ref{Prop:WienerEquiv}$'$ in Section \ref{sec2} then
shows that different choices of $r$ give rise to equivalent norm estimates
on short-time Fourier transforms.
%the local component in a coorbit space with global component in
%a Lebesgue space can
%be chosen to be any $L^r$-space with $r\in (0,\infty ]$ when estimating
%the short-time Fourier transforms.
Again we remark that if $r$ belongs
to the subset $[1,\infty ]$ of $(0,\infty ]$ and that all involved spaces
are Banach spaces, then our results can be obtained in other less complicated
ways, given in e.{\,}g. Chapters 11 and 12 in \cite{Gc2}.

\par

\section*{Acknowledgement}

\par

I am very grateful to Professor Hans Feichtinger for reading parts of the paper
and giving valuable comments, leading to improvements of the content
and the style.

\par

%%%%%%%%%%%%%%%%%%%%%%
\section{Preliminaries}\label{sec1}
%%%%%%%%%%%%%%%%%%%%%%

\par

In this section we recall some basic facts. We start by discussing
Gelfand-Shilov spaces and their properties. Thereafter we recall
some properties of modulation spaces and discuss different aspects
of periodic distributions

\par

\subsection{Gelfand-Shilov spaces and Gevrey classes}\label{subsec1.1}
%We start by recalling some facts on Gelfand-Shilov spaces.
Let $0<s,\sigma \in \mathbf R$ be fixed. Then the Gelfand-Shilov
space $\mathcal S_{s}^\sigma (\rr d)$
($\Sigma _{s}^\sigma (\rr d)$) of Roumieu type (Beurling type) with parameters $s$
and $\sigma$ consists of all $f\in C^\infty (\rr d)$ such that
\begin{equation}\label{gfseminorm}
\nm f{\mathcal S_{s,h}^\sigma }\equiv \sup \frac {|x^\alpha \partial ^\beta
f(x)|}{h^{|\alpha  + \beta |}\alpha !^s \, \beta !^\sigma}
\end{equation}
is finite for some $h>0$ (for every $h>0$). Here the supremum should be taken
over all $\alpha ,\beta \in \mathbf N^d$ and $x\in \rr d$. We equip
$\mathcal S_{s}^\sigma (\rr d)$ ($\Sigma _{s}^\sigma (\rr d)$) by the canonical inductive limit
topology (projective limit topology) with respect to $h>0$, induced by
the semi-norms in \eqref{gfseminorm}.

\par

\medspace

The \emph{Gelfand-Shilov distribution spaces} $(\mathcal S_s^{\sigma})'(\rr d)$
and $(\Sigma _s^\sigma)'(\rr d)$ are the dual spaces of $\mathcal S_s^{\sigma}(\rr d)$
and $\Sigma _s^\sigma (\rr d)$, respectively.  As for the Gelfand-Shilov spaces there 
is a canonical projective limit topology (inductive limit topology) for $(\maclS _s^{\sigma})'(\rr d)$ 
($(\Sigma _s^\sigma)'(\rr d)$).(Cf. \cite{GS, Pil1, Pil3}.)
For conveniency we set
$$
\maclS _s=\maclS _s^s,\quad \maclS _s'=(\maclS _s^s)',\quad
\Sigma _s=\Sigma _s^s
\quad \text{and}\quad
\Sigma _s'=(\Sigma _s^s)'.
$$

\par

From now on we let $\mathscr F$ be the Fourier transform which
takes the form
$$
(\mathscr Ff)(\xi )= \widehat f(\xi ) \equiv (2\pi )^{-\frac d2}\int _{\rr
{d}} f(x)e^{-i\scal  x\xi }\, dx
$$
when $f\in L^1(\rr d)$. Here $\scal \cdo \cdo$ denotes the usual
scalar product on $\rr d$. The map $\mathscr F$ extends 
uniquely to homeomorphisms on $\mathscr S'(\rr d)$,
from $(\mathcal S_s^\sigma )'(\rr d)$ to $(\mathcal S_\sigma ^s)'(\rr d)$ and
from $(\Sigma _s^\sigma)'(\rr d)$ to $(\Sigma _\sigma ^s)'(\rr d)$. Furthermore,
$\mascF$ restricts to
homeomorphisms on $\mathscr S(\rr d)$, from
$\mathcal S_s^\sigma (\rr d)$ to $\mathcal S_\sigma ^s(\rr d)$ and
from $\Sigma _s^\sigma (\rr d)$ to $\Sigma _\sigma ^s(\rr d)$,
and to a unitary operator on $L^2(\rr d)$. 

\par

\par

Next we consider Gevrey classes on $\rr d$. Let $\sigma \ge 0$.
For any compact set $K\subseteq \rr d$, $h>0$ and $f\in C^{\infty}(K)$ let
\begin{equation}\label{e2}
\nm {f}{K,h,\sigma} \equiv \underset{\alpha\in \nn d}\sup
\left (
\frac{\nm {\partial^{\alpha}f}{L^{\infty}(K)}}{h^{\vert \alpha\vert}\alpha ! ^\sigma}
\right ) .
\end{equation} 
The Gevrey class $\maclE _\sigma (K)$ ($\maclE _{0,\sigma}(K)$) of order $\sigma$ and
of Roumieu type (of Beurling type) is the set of all
$f\in C^{\infty}(K)$ such that \eqref{e2} is finite for some (for every)
$h>0$. We equipp $\maclE _\sigma (K)$ ($\maclE _{0,\sigma}(K)$) by
the inductive (projective) limit topology with respect to $h>0$,
supplied by the seminorms in \eqref{e2}. Finally if $\lbrace
K_j\rbrace_{j\geq 1}$ is an exhausted
sets of compact subsets of $\rr d$, then let
\begin{alignat*}{3}
\maclE _\sigma (\rr d) &= \projlim{j} \maclE _\sigma (K_j)
& \quad &\text{and} &\quad
\maclE _{0,\sigma}(\rr d) &= \projlim{j} \maclE _{0,\sigma}(K_j).
\intertext{In particular,}
\maclE _\sigma (\rr d) &=\underset{j\geq 1}\bigcap \maclE _\sigma (K_j)
& \quad &\text{and} &\quad
\maclE _{0,\sigma}(\rr d) &= \underset{j\geq 1}\bigcap \maclE _{0,\sigma}(K_j).
\end{alignat*}
It is clear that $\maclE _{0,0}(\rr d)$ contains all constant functions
on $\rr d$, and that $\maclE _{0}(\rr d)\setminus \maclE _{0,0}(\rr d)$
contains all non-constant trigonometric polynomials.

\par

\subsection{Ordered, dual and phase split bases}\label{subsec1.2}

\par

Our discussions involving Zak transforms, periodicity,
modulation spaces and Wiener spaces are done in terms of suitable bases.

\par

\begin{defn}\label{Def:OrdBasis}
Let $E = \{ e_1,\dots,e_d \}$ be an \emph{ordered} basis of
$\rr {d}$. Then $E'$ denotes the basis of $e_1',\dots,e_d'$ 
in $\rr {d}$ which satisfies
$$
\scal {e_j} {e'_k}=2\pi \delta_{jk}
\quad \text{for every}\quad
j,k =1,\dots, d.
$$
The corresponding lattices are given by
\begin{align*}
\Lambda _E
&=
\sets{n_1e_1+\cdots+n_de_d}{(n_1,\dots,n_d)\in \zz d},
\intertext{and}
\Lambda'_E
&=
\Lambda_{E'}=\sets{\nu _1e'_1+\cdots+\nu _de'_d}{(\nu _1,\dots,\nu _d)
\in \zz d}.
\end{align*}
The sets $E'$ and $\Lambda '_E$ are
called the dual basis and dual lattice of $E$ and $\Lambda _E$, respectively.
If $E_1,E_2$ are ordered bases of $\rr d$ such that a permutation of
$E_2$ is the dual basis for $E_1$, then the pair $(E_1,E_2)$ are called
\emph{permuted dual bases} (to each others on $\rr d$).
\end{defn}

\par

\begin{rem}\label{Rem:TEMap}
Evidently, if $E$ is the same as in Definition \ref{Def:OrdBasis},
then there is a matrix $T_E$ with $E$ as the image of
the standard basis in $\rr d$. Then $E'$ is the image of the standard basis
under the map $T_{E'}= 2\pi(T^{-1}_E)^t$.
\end{rem}

\par

Two ordered bases on $\rr d$ can be used to construct a uniquely defined
ordered basis for $\rr {2d}$ as in the following definition.

\par

\begin{defn}\label{Def:FromTwoToOneBasis}
Let $E_1,E_2$ be ordered bases of $\rr d$,
$$
V_1=\sets {(x,0)\in \rr {2d}}{x\in \rr d},
\quad
V_2=\sets {(0,\xi )\in \rr {2d}}{\xi \in \rr d}
$$
and let $\pi _j$ from $\rr {2d}$ to $\rr d$, $j=1,2$, be the projections
$$
\pi _1(x,\xi ) = x
\quad \text{and}\quad
\pi _2(x,\xi ) = \xi .
$$
Then $E_1\times E_2$ is the ordered basis $\{ e_1,\dots ,e_{2d}\}$
of $\rr {2d}$ such that
\begin{alignat*}{2}
\{ e_1,\dots ,e_d\} &\subseteq V_1,
&\qquad
E_1 &= \{ \pi _1(e_1),\dots ,\pi _1(e_d)\} ,
\\[1ex]
\{ e_{d+1},\dots ,e_{2d}\}
&\subseteq V_2 &
\quad \text{and}\quad
E_2 &= \{ \pi _2(e_{d+1}),\dots ,\pi _2(e_{2d})\} .
\end{alignat*}
\end{defn}

\par

In the phase space it is convenient to consider phase split bases, which
are defined as follows.

\par

\begin{defn}\label{Def:Phasesplit}
Let $V_1$, $V_2$, $\pi _1$ and $\pi _2$ be as in Definition
\ref{Def:FromTwoToOneBasis}, $E$ be an ordered basis of
the phase space $\rr {2d}$ and let
$E_0\subseteq E$. Then
$E$ is called \emph{phase split} (with respect to $E_0$),
if the following is true:
\begin{enumerate}
\item the span of $E_0$ and $E\setminus E_0$ equal
$V_1$ and $V_2$, respectively;

\vrum

\item let $E_1=\pi _1(E_0)$ and $E_2=
\pi _2(E\setminus E_0)$ be the bases in $\rr d$ which preserves the
orders from $E_0$ and $E\setminus E_0$. Then $(E_1,E_2)$
are permuted dual bases.
\end{enumerate}
If $E$ is a phase split basis with respect to $E_0$ and that $E_0$ consists of
the first $d$ vectors in $E$, then $E$ is called \emph{strongly phase split}
(with respect to $E_0$).
\end{defn}

\par

In Definition \ref{Def:Phasesplit} it is understood that the orderings of
$E_0$ and $E\setminus E_0$
are inherited from the ordering in $E$.

\par

\begin{rem}\label{Rem:CommentPhasSplit}
Let $E$ and $E_j$, $j=0,1,2$ be the same as in Definition \ref{Def:Phasesplit}.
It is evident that $E_0$ and $E\setminus E_0$  consist of $d$ elements, and
that $E_1$ and $E_2$ are uniquely defined. The pair $(E_1,E_2)$ is called
\emph{the pair of permuted dual bases, induced by $E$ and $E_0$}.

\par

On the other hand, suppose that $(E_1,E_2)$ is a pair of permuted dual bases to each
others on $\rr d$. Then it is clear that for
$E_1\times E_2=\{ e_1,\dots ,e_{2d}\}$ in Definition \ref{Def:FromTwoToOneBasis}
and $E_0=\{ e_1,\dots ,e_d\}$, we have that
$E_0$ and $E$ fullfils all properties in Definition \ref{Def:Phasesplit}. In this case,
$E_1\times E_2$ is called the phase split basis (of $\rr {2d}$) induced by $(E_1,E_2)$.

\par

It follows that if $E'$, $E_1'$ and $E_2'$ are the dual bases of $E$, $E_1$ and $E_2$,
repsectively, then $E'=E_1'\times E_2'$.
\end{rem}

\par

\subsection{Invariant quasi-Banach spaces and spaces of
mixed quasi-normed spaces of Lebesgue types}\label{subsec1.3}

\par

We recall that a quasi-norm $\nm {\cdo}{\mascB}$ of order $r \in (0,1]$ on the
vector-space $\mascB$ over $\mathbf C$ is a nonnegative functional on
$\mascB$ which satisfies
\begin{alignat}{2}
 \nm {f+g}{\mascB} &\le 2^{\frac 1r-1}(\nm {f}{\mascB} + \nm {g}{\mascB}), &
\quad f,g &\in \mascB ,
\label{Eq:WeakTriangle1}
\\[1ex]
\nm {\alpha \cdot f}{\mascB} &= |\alpha| \cdot \nm f{\mascB},
& \quad \alpha &\in \mathbf{C},
\quad  f \in \mascB
\notag
\intertext{and}
   \nm f {\mascB} &= 0\quad  \Leftrightarrow \quad f=0. & &
\notag
\end{alignat}
The space $\mascB$ is then called a quasi-norm space. A complete
quasi-norm space is called a quasi-Banach space. If $\mascB$
is a quasi-Banach space with
quasi-norm satisfying \eqref{Eq:WeakTriangle1}
then by \cite{Aik,Rol} there is an equivalent quasi-norm to $\nm \cdo {\mascB}$
which additionally satisfies
\begin{align}\label{Eq:WeakTriangle2}
\nm {f+g}{\mascB}^r \le \nm {f}{\mascB}^r + \nm {g}{\mascB}^r, 
\quad f,g \in \mascB .
\end{align}
From now on we always assume that the quasi-norm of the quasi-Banach space $\mascB$
is chosen in such way that both \eqref{Eq:WeakTriangle1} and \eqref{Eq:WeakTriangle2}
hold.

\par

Before giving the definition of $v$-invariant spaces, we recall some facts on weight
functions.

\par

A \emph{weight} or \emph{weight function} on $\rr d$ is a
positive function $\omega
\in  L^\infty _{loc}(\rr d)$ such that $1/\omega \in  L^\infty _{loc}(\rr d)$.
The weight $\omega$ is called \emph{moderate},
if there is a positive weight $v$ on $\rr d$ such that
\begin{equation}\label{moderate}
\omega (x+y) \lesssim \omega (x)v(y),\qquad x,y\in \rr d.
\end{equation}
If $\omega$ and $v$ are weights on $\rr d$ such that
\eqref{moderate} holds, then $\omega$ is also called
\emph{$v$-moderate}.
We note that \eqref{moderate} implies that $\omega$ fulfills
the estimates
\begin{equation}\label{moderateconseq}
v(-x)^{-1}\lesssim \omega (x)\lesssim v(x),\quad x\in \rr d.
\end{equation}
We let $\mascP _E(\rr d)$ be the set of all moderate weights on $\rr d$.

\par

It can be proved that if $\omega \in \mascP _E(\rr d)$, then
$\omega$ is $v$-moderate for some $v(x) = e^{r|x|}$, provided the
positive constant $r$ is large enough (cf. \cite{Gc2.5}). In particular,
\eqref{moderateconseq} shows that for any $\omega \in \mascP
_E(\rr d)$, there is a constant $r>0$ such that
$$
e^{-r|x|}\lesssim \omega (x)\lesssim e^{r|x|},\quad x\in \rr d.
$$

\par

We say that $v$ is
\emph{submultiplicative} if $v$ is even and \eqref{moderate}
holds with $\omega =v$. In the sequel, $v$ and $v_j$ for
$j\ge 0$, always stand for submultiplicative weights if
nothing else is stated. The next definition is similar to \cite[Section 3]{FG1}
in the Banach space case.

\par

\begin{defn}\label{Def:InvSpaces}
Let $r\in (0,1]$, $v\in \mascP _E(\rr {d})$ and let $\mascB = \mascB (\rr d)
\subseteq L^r_{loc}(\rr {d})$ be a quasi-Banach space such that $\Sigma _1
(\rr d)\subseteq \mascB (\rr d)$. Then $\mascB$ is called $v$-invariant on
$\rr d$ if the following is true:
\begin{enumerate}
\item $x\mapsto f(x+y)$ belongs to $\mascB$ for every $f\in \mascB$
and $y\in \rr {d}$.

\vrum

\item There is a constant $C>0$ such that $\nm {f_1}{\mascB}\le C\nm {f_2}{\mascB}$
when $f_1,f_2\in \mascB$ are such that $|f_1|\le |f_2|$. Moreover,
$$
\nm {f(\cdo +y)}{\mascB}\lesssim \nm {f}{\mascB}v(y),\qquad
f\in \mascB ,\ y\in \rr {d}.
$$
\end{enumerate}
\end{defn}

\par

Let $\mascB$ be as in Definition \ref{Def:InvSpaces}, $E$ be a basis for
$\rr d$ and let $\kappa (E)$ be the closed parallelepiped spanned by $E$.
The \emph{discrete version, $\ell _{\mascB ,E} = \ell _{\mascB ,E}(\Lambda _E)$,
of $\mascB$} with respect to $E$
is the set of all $a\in \ell _0'(\Lambda _E)$ such that
$$
\nm a{\ell _{\mascB ,E}} \equiv \Nm {\sum _{j\in \Lambda _E}a(j)\chi _{j+\kappa (E)}}{\mascB}
$$
is finite.

\par

An important example on $v$-invariant spaces concerns mixed quasi-norm
spaces of Lebesgue type, given in the following definition.

\par

\begin{defn}\label{Def:DiscLebSpaces}
Let $E = \{  e_1,\dots ,e_d \}$ be an ordered basis of $\rr d$, $\kappa (E)$ be the
parallelepiped spanned by $E$, $\omega \in \mascP _E(\rr d)$
$\mabfq =(q_1,\dots ,q_d)\in (0,\infty ]^{d}$ and $r=\min (1,\mabfq )$.
If  $f\in L^r_{loc}(\rr d)$, then
$$
\nm f{L^{\mabfq }_{E,(\omega )}}\equiv
\nm {g_{d-1}}{L^{q_{d}}(\mathbf R)}
$$
where  $g_k\, :\, \rr {d-k}\to \mathbf R$,
$k=0,\dots ,d-1$, are inductively defined as
\begin{align*}
g_0(x_1,\dots ,x_{d})
&\equiv
|f(x_1e_1+\cdots +x_{d}e_d)\omega (x_1e_1+\cdots +x_{d}e_d)|,
\\[1ex]
\intertext{and}
g_k(\mabfz _k) &\equiv
\nm {g_{k-1}(\cdo ,\mabfz_k)}{L^{q_k}(\mathbf R)},
\quad \mabfz _k\in \rr {d-k},\ k=1,\dots ,d-1.
\end{align*}
If $\Omega \subseteq \rr d$ is measurable,
then $L^{\mabfq }_{E,(\omega )}(\Omega )$ consists
of all $f\in L^r_{loc}(\Omega )$ with finite quasi-norm
$$
\nm f{L^{\mabfq}_{E,(\omega )}(\Omega )}
\equiv
\nm {f_\Omega }{L^{\mabfq}_{E,(\omega )}(\rr d)},
\qquad
f_\Omega (x)
\equiv
\begin{cases}
f(x), &\text{when}\ x\in \Omega
\\[1ex]
0, &\text{when}\ x\notin \Omega .
\end{cases}
$$
The space $L^{\mabfq }_{E,(\omega )}(\Omega )$ is called 
\emph{$E$-split Lebesgue space (with respect to $\omega$, $\mabfq$
and $\Omega$)}.
\end{defn}

\par

We let $\ell ^\mabfp _{E,(\omega )}(\Lambda _E)$ be the discrete
version of $\mascB = L^\mabfp _{E,(\omega )}(\rr d)$ when
$\mabfp \in (0,\infty ]^d$.

\par

Suppose that $E$ and $\Lambda$ are the same as in Definition
\ref{Def:DiscLebSpaces}. Then we let $(\ell _E^0)'(\Lambda )$
be the set of
all formal sequences $\{ a(j)\} _{j\in \Lambda}$, and we let
$\ell _E^0(\Lambda )$
be the set of all such sequences such that at most finite
numbers of $a(j)$ are
non-zero.

\par

\begin{rem}\label{Rem:LebExpIdent}
Evidently, $L^{\mabfq}_{E,(\omega )} (\Omega )$ and
$\ell ^{\mabfq}_{E,(\omega )} (\Lambda )$
in Definition \ref{Def:DiscLebSpaces} are quasi-Banach spaces of order
$\min (\mabfp,1)$. We set
$$
L^{\mabfq}_{E} = L^{\mabfq}_{E,(\omega )}
\quad \text{and}\quad
\ell ^{\mabfq}_{E} = \ell ^{\mabfq}_{E,(\omega )}
$$
when $\omega =1$. For conveniency we identify
$\mabfq = (q,\dots ,q)\in (0,\infty ]^d$
with $q\in (0,\infty ]$ when considering spaces involving
Lebesgue exponents. In particular,
\begin{alignat*}{5}
L^{q}_{E,(\omega )} &=  L^{\mabfq}_{E,(\omega )},
&\quad
L^{q}_{E} &= L^{\mabfq}_{E},
&\quad
\ell ^{q}_{E,(\omega )} &= \ell ^{\mabfq}_{E,(\omega )}
&\quad &\text{and} &\quad 
\ell ^{q}_{E} &= \ell ^{\mabfq}_{E}
\intertext{for such $\mabfq$, and notice that these spaces agree with}
&{\phantom =}L^q_{(\omega )}, &\qquad
&{\phantom =}L^q, &\qquad
&{\phantom =}\ell ^{q}_{(\omega )}
&\quad &\text{and}&\quad
&{\phantom =}\ell ^{q},
\end{alignat*}
respectively, with equivalent quasi-norms.
\end{rem}

\par

\subsection{Modulation and Wiener spaces}\label{subsec1.4}

\par

We consider a general class of modulation spaces
given in the following definition (cf. \cite{Fei5}).

\par

\begin{defn}\label{Def:ModSpaces}
Let $\omega ,v\in \mascP _E(\rr {2d})$ be such that $\omega$ is $v$-moderate,
$\mascB$ be a $v$-invariant quasi-Banach space on
$\rr {2d}$, and let $\phi \in \maclS _{1/2}(\rr d)\setminus 0$.
Then the \emph{modulation space} $M(\omega ,\mascB)$ consists of all $f\in \maclS _{1/2}'
(\rr d)$ such that
\begin{equation}\label{Eq:ModSpNorm}
\nm f{M(\omega ,\mascB)} \equiv \nm {V_\phi f \cdot \omega}{\mascB}
\end{equation}
is finite.
\end{defn}

\par

An important family of modulation spaces which contains
the classical modulation spaces, introduced by Feichtinger in
\cite{F1}, is given next.

\par

\begin{defn}\label{Def:ExtClassModSpaces}
Let $\mabfp ,\mabfq \in (0,\infty ]^d$, $E_1$ and $E_2$ be ordered bases
of $\rr d$, $E=E_1\times E_2$,
$\phi \in \Sigma _1(\rr d)\setminus 0$ and let
$\omega \in \mascP _E(\rr {2d})$.
For any $f\in \Sigma _1'(\rr d)$ set
\begin{multline*}
\nm f{M^{\mabfp ,\mabfq}_{E,(\omega )}}
\equiv
\nm {H_{1,f,E_1,\mabfp ,\omega }}{L^{\mabfq}_{E_2}},
\\[1ex]
\quad \text{where}\quad
H_{1, f,E_1,\mabfp ,\omega }(\xi )
\equiv
\nm {V_\phi f (\cdo ,\xi )\omega (\cdo ,\xi )}{L^{\mabfp}_{E_1}}
\end{multline*}
and
\begin{multline*}
\nm f{W^{\mabfp ,\mabfq}_{E,(\omega )}}
\equiv
\nm {H_{2,f,E_2,\mabfq ,\omega }}{L^{\mabfp}_{E_1}},
\\[1ex]
\quad \text{where}\quad
H_{2, f,E_2,\mabfq ,\omega }(x)
\equiv
\nm {V_\phi f (x,\cdo )\omega (x,\cdo )}{L^{\mabfq}_{E_2}}
\end{multline*}
The \emph{modulation space}
$M^{\mabfp ,\mabfq}_{E,(\omega )}(\rr d)$
($W^{\mabfp ,\mabfq}_{E,(\omega )}(\rr d)$)
consist of all $f\in \Sigma _1'(\rr d)$ such that
$\nm f{M^{\mabfp ,\mabfq}_{E,(\omega )}}$ 
($\nm f{W^{\mabfp ,\mabfq}_{E,(\omega )}}$) is finite.
\end{defn}

\par

The theory of modulation spaces has developed in different ways since they
were introduced in \cite{F1} by Feichtinger. (Cf. e.{\,}g. \cite{Fei5,GaSa,Gc2,Toft15}.)
For example, let $\mabfp$, $\mabfq$, $E$, $\omega$ and $v$ be the same
as in Definition \ref{Def:ModSpaces} and \ref{Def:ExtClassModSpaces},
and let $\mascB = L^{\mabfp ,\mabfq}_E(\rr {2d})$ and $r=\min (1,\mabfp ,\mabfq)$.
Then $M(\omega ,\mascB)=M^{\mabfp ,\mabfq}_{E,(\omega )}(\rr d)$
is a quasi-Banach space. Moreover, $f\in M^{\mabfp ,\mabfq}_{E,(\omega )}(\rr d)$
if and only if $V_\phi f \cdot \omega \in L^{\mabfp ,\mabfq}_{E}(\rr {2d})$,
and different choices of $\phi$ give rise to equivalent quasi-norms in
Definition \ref{Def:ExtClassModSpaces}.
We also note that for any such $\mascB$, then
$$
\Sigma _1(\rr d) \subseteq M^{\mabfp ,\mabfq}_{E,(\omega )}(\rr d) \subseteq \Sigma _1'(\rr d).
$$
Similar facts hold for the space $W^{\mabfp ,\mabfq}_{E,(\omega )}(\rr d)$.
(Cf. \cite{GaSa,Toft15}.)

\par

We shall consider various kinds of Wiener spaces involved later on when
finding different characterizations of modulation spaces. The following
type of Wiener spaces can essentially be found in e.{\,}g.
\cite{FG1,GaSa,Gc2}, and is related to coorbit spaces of Lebesgue spaces. 

\par

\begin{defn}\label{Def:WienerSpace}
Let $\mabfr \in (0,\infty ]^d$, 
$\omega _0\in \mascP _E(\rr d)$, $\omega \in \mascP _E(\rr {2d})$,
$\phi \in \Sigma _1(\rr d)\setminus 0$, %$\mascB \subseteq L^r_{loc}(\rr d)$,
$E\subseteq \rr d$ be an ordered basis, and let $\kappa (E)$ be the closed parallelepiped
spanned by $E$. Also let $\mascB = \mascB (\rr d)$ and $\mascB _0 = \mascB _0(\rr d)$
be invariant QBF-spaces on $\rr d$, $f$ and $F$ be measurable on $\rr d$ respective
$\rr {2d}$, $F_\omega =F\cdot \omega$, and let
$\ell _{\mascB ,E}(\Lambda _E)$ be the discrete version of $\mascB$ with respect to $E$.
\begin{enumerate}
\item Then $\nm f{\sfW ^{\mabfr}_{E}(\omega _0,\ell _{\mascB ,E} )}$ is given by
\begin{align*}
\nm f{\sfW ^{\mabfr}_{E}(\omega _0,\ell _{\mascB ,E} )}
&\equiv
\nm {h_{E,\omega _0,\mabfq ,f}}{\ell _{\mascB ,E}(\Lambda _E)},
\\[1ex]
h_{E,\omega _0,\mabfq ,f} (j)
&=
\nm {f} {L_E^{\mabfr}(j+\kappa (E))}\omega _0(j),
\quad
j\in \Lambda _E.
\end{align*}
The set $\sfW ^{\mabfr}_{E}(\omega ,\ell _{\mascB ,E} )$ consists of all measurable
$f$ on $\rr d$ such that $\nm f{\sfW ^{\mabfr}_{E}(\omega _0,\ell _{\mascB ,E} )}<\infty$;

\vrum

\item Then $\nm F{\sfW ^{\mabfr}_{k,E}(\omega ,\ell _{\mascB ,E}\,,\mascB _0 )}$, $k=1,2$,
are given by
\begin{alignat*}{2}
\nm F{\sfW ^{\mabfr}_{1,E}(\omega ,\ell _{\mascB ,E}\,,\mascB _0 )}
&\equiv
\nm {\fy _{F,\omega ,\mabfr ,\mascB ,E}}{\mascB _0},&\quad
\fy _{F,\omega ,\mabfr ,\mascB ,E}(\xi )
&=
\nm {F_\omega (\cdo ,\xi )}{\sfW ^{\mabfr}_{E}(1,\ell _{\mascB ,E} )},
\intertext{and}
\nm F{\sfW ^{\mabfr}_{2,E}(\omega ,\ell _{\mascB ,E}\, ,\mascB _0 )}
&\equiv
\nm {\psi _{F,\omega ,\mascB _0}}{\sfW ^{\mabfr}_{E}
(1,\ell _{\mascB ,E} )},&
\qquad
\psi _{F,\omega ,\mascB _0}(x) &= \nm {F_\omega (x,\cdo )}{\mascB _0}.
\end{alignat*}
The set $\sfW ^{\mabfr}_{k,E}(\omega ,\ell _{\mascB ,E}\,,\mascB _0 )$ 
consists of all
measurable $F$ on $\rr {2d}$ such that
$\nm F{\sfW ^{\mabfr}_{k,E}
(\omega ,\ell _{\mascB ,E}\,,\mascB _0 )}<\infty$, $k=1,2$.
\end{enumerate}
\end{defn}

\par

The space $\sfW ^{\mabfr}_{E}(\omega _0,\ell _{\mascB ,E})$
in Definition
\ref{Def:WienerSpace} is essentially a Wiener amalgam space with
$L^{\mabfr}_E$ as local (quasi-)norm and $\mascB$ or
$\ell _{\mascB ,E}(\Lambda _E)$
as global component. They are also related to
coorbit spaces. (See \cite{Fe0,FG1,FG2,FeLu,Rau1,Rau2}.)

\par

In fact, $\sfW ^\infty (\omega _0,\ell ^{\mabfp})$ in Definition
\ref{Def:WienerSpace} (i.{\,}e. the case
$\mabfr = (\infty ,\dots ,\infty )$ and $E$ is the standard basis)
is the \emph{coorbit space} of
$L^{\mabfp}(\rr d)$ with weight $\omega _0$, and is
sometimes denoted by
$$
\Co (L^{\mabfp}_{(\omega _0)}(\rr d))\quad \text{or}\quad 
W(L_{(\omega _0)} ^{\mabfp}) = W(L_{(\omega)}
^{\mabfp}(\rr d)),
$$
in the literature (cf. \cite{Gc2,Rau1,Rau2}).

\par

\begin{rem}\label{Rem:PerWienerSpace}
Let $\mabfp$, $\omega _0$, $\omega$, $E$, $\mascB$,
$\mascB _0$, $f$ and $F$
be the same as in Definition \ref{Def:WienerSpace}.
Evidently, by using the fact that $\omega _0$ is
$v_0$-moderate for some $v_0$,
it follows that
\begin{align*}
\nm {f\cdot \omega _0}{\sfW ^{\mabfq}_{E}(1,\ell _{\mascB ,E})}
&\asymp
\nm {f}{\sfW ^{\mabfq}_{E}(\omega _0,\ell _{\mascB ,E})}
\intertext{and}
\nm {F\cdot \omega}
{\sfW ^{\mabfq}_{k,E}(1,\ell _{\mascB ,E}\, ,\mascB _0 )}
&=
\nm {F}{\sfW ^{\mabfq}_{k,E}(\omega ,\ell
_{\mascB ,E}\, ,\mascB _0)}
\end{align*}
for $k=1,2$. Furthermore,
\begin{alignat*}{1}
\sfW ^{\mabfq}_{1,E}(\omega ,\ell _{\mascB ,E}\, ,\mascB _0)
&=
\omega ^{-1}
\cdot
\sfW ^{\mabfq}_{E}(1,\ell _{\mascB ,E}\, ;\, \mascB _0)
\intertext{and}
\sfW ^{\mabfq}_{2,E}(\omega ,\ell _{\mascB ,E}\, ,\mascB _0)
&=
\omega ^{-1}
\cdot
\mascB _0 (\rr d\, ;\, \sfW ^{\mabfq}_{E}
(1,\ell _{\mascB ,E} ) ).
\end{alignat*}
Here and in what follows,
$\mascB (\rr {d}\, ;\, \mascB _0)
=
\mascB (\rr {d}\, ;\, \mascB _0(\rr {d_0}))$
is the set of all functions $g$ in $\mascB $ with values in
$\mascB _0$, which are equipped with the quasi-norm
$$
\nm g{\mascB (\rr {d};\mascB _0)}\equiv \nm {g_0}{\mascB},
\qquad g_0(x) \equiv \nm {g(x)}{\mascB _0},
$$
when $\mascB (\rr {d})$ and $\mascB _0(\rr {d_0})$ are
invariant QBF-spaces.
\end{rem}

\par

Later on we discuss periodicity in the framework of certain
modulation spaces which are related to spaces which are defined
by imposing $L^\infty$-conditions on the configuration
variable of corresponding short-time Fourier transforms.

\par

\begin{defn}\label{Def:AltModSpace}
Let $E$, $\mabfr$, $\mascB _0$ and $\omega \in \mascP _E(\rr {2d})$
be the same as in Definition \ref{Def:WienerSpace}, and let $\phi \in
\Sigma _1(\rr d)\setminus 0$. Then
$\splM ^{\mabfr}_E(\omega ,\mascB _0)$ and
$\splW ^{\mabfr}_E(\omega ,\mascB _0)$)
are the sets of all $f\in \Sigma _1'(\rr d)$ such that
\begin{align*}
\nm f{\splM ^{\mabfr}_E(\omega ,\mascB _0)}
&\equiv
\nm {V_\phi f}{\sfW ^{\mabfr}_{1,E}(\omega ,\ell ^{\infty}_E,\mascB _0)}
\quad
\intertext{respectively}
\nm f{\splW ^{\mabfr}_E(\omega ,\mascB _0)}
&\equiv
\nm {V_\phi f}{\sfW ^{\mabfr}_{2,E}(\omega ,\ell ^{\infty}_E,\mascB _0)}
\end{align*}
are finite.
\end{defn}

\par

\begin{rem}\label{Rem:WienerSpace}
For the spaces in Definition \ref{Def:WienerSpace}
we set $\sfW ^{q_0,\mabfr _0} =  \sfW ^{\mabfr}$, when
\begin{alignat*}{2}
\mabfr _0 &=(r_1,\dots ,r_d)\in (0,\infty ]^d,&
\quad \text{and}\quad
\mabfr &= (q_0,\dots ,q_0,r_1,\dots ,r_d)\in (0,\infty ]^{2d},
\end{alignat*}
and similarly for other types of exponents and for the spaces
% $\splM ^{\mabfr}_E(\omega ,\mascB _0)$
% and
% $\splW ^{\mabfr}_E(\omega ,\mascB _0)$
% $$
% \splM ^{\mabfr}_E(\omega _0,\mascB _0),
% \quad \text{and}\quad
% \splW ^{\mabfr}_E(\omega _0,\mascB _0),
% $$
in Definitions \ref{Def:ExtClassModSpaces}
and \ref{Def:AltModSpace}. (See also Remark
\ref{Rem:LebExpIdent}.)
We also set
$$
M^{\infty ,\mabfq}_{E,(\omega )}
=
M^{\infty ,\mabfq}_{E_2,(\omega )}
\quad \text{and}\quad
W^{\infty,\mabfq}_{E,(\omega )}
=
W^{\infty,\mabfq}_{E_2,(\omega )}
$$
when $E_1,E_2$ are ordered bases of $\rr d$ and
$E=E_1\times E_2$, for spaces in Definition \ref{Def:ExtClassModSpaces},
since these spaces are independent of $E_1$.
\end{rem}

\par

In Section \ref{sec2} we prove that if $\mascB _0$ is an $E$-split
Lebesgue space on $\rr d$ and $\omega (x,\xi )\in \mascP _E(\rr {2d})$
which is constant with respect to the $x$ variable, then
$\splM ^{\mabfr}_E(\omega ,\mascB _0)$
and $\splW ^{\mabfr}_E(\omega ,\mascB _0)$ are independent of $\mabfr$
and agree with modulation spaces of the form in
Definition \ref{Def:ModSpaces}
(cf. Proposition \ref{Prop:SpecCaseWienerRel1}).

\par

The next result is a reformulation of \cite[Proposition 3.4]{Toft15}, and
indicates how Wiener spaces are connected to modulation spaces. The proof
is therefore omitted. Here, let
\begin{equation}\label{Eq:SubMultMod}
(\Theta _\rho v)(x,\xi )=v(x,\xi )\eabs {x,\xi}^\rho ,
\quad \text{where}\quad
\rho \ge 2d\left ( \frac 1r-1 \right ),
\end{equation}
for any submultiplicative $v\in \mascP _E(\rr {2d})$ and $r\in (0,1]$.
It follows that $L^1_{(\Theta _\rho v)}(\rr {2d})$ is
continuously embedded in
$L^r_{(v)}(\rr {2d})$, giving that $M^1_{(\Theta _\rho v)}(\rr d)
\subseteq M^r_{(v)}(\rr d)$. Hence if $\phi \in M^1_{(\Theta _\rho v)}
\setminus 0$, $\ep _0$ is chosen such that $S^\Lambda _{\phi ,\phi}$
is invertible on $M^1_{(\Theta _\rho v)}(\rr d)$ for every
$\Lambda =\ep \Lambda _E$, $\ep \in (0,\ep _0]$, it follows that both
$\phi$ and its canonical dual with respect to $\Lambda$ belong to
$M^r_{(v)}(\rr d)$. Notice that such $\ep _0>0$ exists in view of
\cite[Theorem S]{Gc1}.

\par

\begin{prop}\label{Prop:WienerEquiv}
Let $E$ be a phase split basis for $\rr {2d}$,
$\mabfp \in (0,\infty ]^{2d}$, $r=\min (1,\mabfp )$,
$\omega ,v\in \mascP _E(\rr {2d})$
be such that $\omega$ is $v$-moderate, $\rho$ and $\Theta _\rho v$ be
as in \eqref{Eq:SubMultMod} with strict inequality when $r<1$,
and let $\phi _1,\phi _2\in M^1_{(\Theta _\rho v)} (\rr d)\setminus 0$.
Then
$$
\nm f{M^{\mabfp}_{E,(\omega )}}
\asymp
\nm {V_{\phi _1} f}{L^\mabfp _{E,(\omega )}}
\asymp 
\nm {V_{\phi _2} f}{\sfW ^\infty _{E}
(\omega ,\ell ^\mabfp _{E})},\quad f\in
\maclS '_{1/2}(\rr d).
$$
In particular, if $f\in \maclS _{1/2}'(\rr d)$, then
$$
f\in M^{\mabfp}_{E,(\omega )}(\rr {2d})
\quad \Leftrightarrow \quad
V_{\phi _1} f \in L^\mabfp _{E,(\omega )}(\rr {2d})
\quad \Leftrightarrow \quad
V_{\phi _2} f \in \sfW ^\infty _{E}(\omega ,\ell ^\mabfp _{E}
(\Lambda _E)).
$$
\end{prop}

\par

In Section \ref{sec2} we extend this result in such way
that we may replace
$\sfW ^\infty _{E}(\omega ,\ell ^\mabfp _{E})$
by $\sfW ^r _{E}(\omega ,\ell ^\mabfp _{E})$ for any $r>0$.

\par

\subsection{Classes of periodic elements}

\par

We consider
spaces of periodic Gevrey functions and their duals.

\par

Let $s,\sigma \in\mathbf{R}_{+}$ be such that $s+t\geq 1$,
$f\in (\mathcal S_{s}^{\sigma})'(\rr d)$, 
$E$ be a basis of $\rr d$ and let $E_0\subseteq E$.
Then $f$ is called \emph{$E_0$-periodic} if $f(x+y)=f(x)$ 
for every $x\in \rr d$ and $y\in E_0$.

\par

We note that for any $\Lambda _E$-periodic function
$f\in C^{\infty}(\rr d)$, we have 
\begin{align}
f &= \sum_{\alpha\in \Lambda ' _E} c(f,{\alpha})e^{i\scal \cdo \alpha},
\label{Eq:Expan2}
\intertext{where $c(f,{\alpha})$ are the Fourier coefficients given by}
c(f,{\alpha}) &\equiv \vert \kappa (E) \vert ^{-1}
(f,e^{i\scal \cdo \alpha})_{L^{2}(E)}.\notag
\end{align}

\par

For any $s\ge 0$ and basis $E\subseteq \rr d$ we let
$\maclE _{0,\sigma}^{E}(\rr d)$
and $\maclE _{\sigma}^{E}(\rr d)$ be the sets of all $E$-periodic 
elements
in $\maclE _{0,\sigma}(\rr d)$ and in $\maclE _{\sigma}(\rr d)$, respectively. Evidently,
$$
\maclE _\sigma^E(\rr d)\simeq \maclE _\sigma (\rr d/\Lambda _E)
\quad \text{and}\quad
\maclE _{0,\sigma}^E(\rr d)\simeq \maclE _{0,\sigma}(\rr d/\Lambda _E),
$$
which is a common approach in the literature.

\par

\begin{rem}
Let $E$ be an ordered basis on $\rr d$
and $V$ be a topological space of functions or (ultra-)distributions on
$\rr d$. Then we use the convention that $V^E$ ($E$ as upper case index) denotes
the $E$ periodic elements in $V$, while $V_E$ ($E$ as lower case index) is the
space analogous to $V$ when $E$ is used as basis.
\end{rem}

\par

Let $s,s_0,\sigma ,\sigma _0>0$ be such that $s+\sigma \ge 1$,
$s_0+\sigma _0\ge 1$ and $(s_0,\sigma _0)\neq (\frac 12,\frac 12)$. Then
we recall that the duals
$(\maclE _\sigma ^E)'(\rr d)$ and $(\maclE _{0,\sigma _0}^E)'(\rr d)$ of
$\maclE _\sigma ^E(\rr d)$ and $\maclE _{0,\sigma _0}^E(\rr d)$, respectively,
can be identified with the $E$-periodic elements in $(\maclS _s^\sigma)'(\rr d)$
and $(\Sigma _{s_0}^{\sigma _0})'(\rr d)$ respectively via unique extension of the form
\begin{align*}
(f,\phi)_{E} &= \sum_{\alpha \in \Lambda '_{E}} c(f,{\alpha})
\overline{c(\phi ,{\alpha})}
%\intertext{or}
%\scal f\phi_{E} &= \sum _{\alpha \in \Lambda _{E}'} c(f,{\alpha}) c(\phi ,{\alpha})
\end{align*}
on $\maclE _{0,\sigma _0}^E (\rr d)\times \maclE _{0,\sigma _0}^E (\rr d)$.
We also let $(\maclE _0^E)'(\rr d)$ be the set of all formal expansions
in \eqref{Eq:Expan2} and $\maclE _0^E(\rr d)$ be the set of
all formal expansions in \eqref{Eq:Expan2} such that at most finite
numbers of $c(f,{\alpha})$ are non-zero (cf. \cite{ToNa}). We refer to
\cite{Pil2,ToNa} for more characterizations of $\maclE _{\sigma}^E$,
$\maclE _{0,\sigma}^E$ and their duals.

\par

%In the next proposition we recall some
%characterizations of periodic
%functions and distributions. We refer to \cite{Pil2} for
%a proof of the first result and to \cite{ToNa}
%for the proof of the second result.
%
%\par
%
%\begin{prop}\label{Prop:PerGevEquiv}
%Let $\sigma >0$, and let $E$ be a basis in $\rr d$. Then
%$\maclE _\sigma^E(\rr d)$ ($\maclE _{0,s}^E(\rr d)$) can
%in a unique way be identified with the set of all expansions
%\eqref{Eq:Expan2} such that
%$$
%|c(f,\alpha )| \lesssim e^{-r|\alpha |^{\frac 1\sigma }}
%$$
%for some $r>0$ (for every $r>0$).
%\end{prop}
%
%\par
%
%\begin{prop}\label{Prop:PerGevDualEquiv}
%Let $\sigma >0$,
%and let $E$ be a basis in $\rr d$. Then $(\maclE _\sigma^E)'(\rr d)$
%($\maclE _{0,\sigma}^E)'(\rr d)$) can in a unique way be identified
%with the set of all expansions
%\eqref{Eq:Expan2} such that
%$$
%|c(f,\alpha )| \lesssim e^{r|\alpha |^{\frac 1\sigma }}
%$$
%for every $r>0$ (for some $r>0$). Moreover,
%if $s>0$ is such that $s+\sigma \ge 1$ ($s+\sigma \ge 1$
%and $(s,\sigma )\neq (\frac 12 ,\frac 12)$),
%then $(\maclE _\sigma ^E)'(\rr d)$ ($\maclE _{0,\sigma}^E)'(\rr d)$)
%agrees with
%the set of $E$-periodic elements in $(\maclS _s^\sigma)'(\rr d)$
%($(\Sigma _s^\sigma)'(\rr d)$).
%\end{prop}

\par

The following definition takes care of spaces of formal
expansions \eqref{Eq:Expan2} with coefficients
obeying specific quasi-norm estimates.

\par

\begin{defn}\label{Def:PerQuasiBanachSpaces}
Let $E$ be a basis of $\rr d$,
$\mascB$ be a
quasi-Banach space  continuously embedded in $\ell _0' (\Lambda ' _E)$
and let $\omega _0$ be a weight on $\rr d$. Then
$\maclL ^E(\omega _0,\mascB )$ consists of all 
$f\in(\maclE_{0}^E)'(\rr d)$ such that 
\begin{equation*}
\nm f{\maclL ^E(\omega _0,\mascB)}
\equiv
\nm {\{ c(f,\alpha)\omega _0(\alpha) \}
_{\alpha \in \Lambda '_E}} {\mascB}
\end{equation*}
is finite.
\end{defn}

\par

%When discussing periodic element in modulation spaces, it is convenient to
%set
%%%
%\begin{alignat}{2}
%M^{\mabfr ,E}(\omega ,\mascB ) &\equiv \splM ^{\mabfr}_E (\omega _0,\mascB )
%\bigcap (\maclE _{0}^{E})'(\rr d),&\quad \omega _0(x,\xi )&=\omega (\xi )
%\label{Eq:PerModSpace1}
%\intertext{and}
%W^{\mabfr ,E}(\omega ,\mascB ) &\equiv \splW ^{\mabfr}_E (\omega _0,\mascB )
%\bigcap (\maclE _{0}^{E})'(\rr d),&\quad \omega _0(x,\xi )&=\omega (\xi )
%\label{Eq:PerModSpace2}
%\end{alignat}
%%%
%when $\omega \in \mascP _E(\rr d)$, $\mabfr \in (0,\infty ]^d$ and $\mascB$
%is a quasi-Banach space contained in $L^r_{loc}(\rr d)$ for some $r>0$.
%We equip these spaces by the induced topology from $\splM ^{\mabfr}_E
%(\omega _0,\mascB )$ and $\splW ^{\mabfr}_E (\omega _0,\mascB )$,
%respectively.
If $\omega _0\in \mascP _E(\rr d)$ and $\omega (x,\xi )=\omega _0(\xi )$,
then
\begin{align}
\begin{aligned}\label{Eq:ModPerNormEquiv1}
\nm f{\splM ^{\mabfr}_E (\omega ,\mascB )}
%&=
%\nm f{M^{\mabfr ,E}(\omega ,\mascB )}
&=
\nm {g\omega _0}{\mascB},
\\[1ex]
\text{when} \quad
g(\xi ) &= \nm {V_\phi f(\cdo ,\xi)}{L^{\mabfr}_E(\kappa (E))},
\quad
%\phantom{{}_{\mascB}}
f\in (\maclE _0^E)'(\rr d),
\end{aligned}
\intertext{and}
\begin{aligned}\label{Eq:ModPerNormEquiv2}
\nm f{\splW ^{\mabfr}_E (\omega ,\mascB )}
%&=
%\nm f{W^{\mabfr ,E}(\omega ,\mascB )}
&=
\nm h{L^{\mabfr}_E(\kappa (E))},
\\[1ex]
\text{when} \quad
h(x) &= \nm {V_\phi f(x,\cdo )\omega _0}{\mascB},
\quad
\phantom{{}_{L^{\mabfr}_E(\kappa (E))}}
\quad f\in (\maclE _0^E)'(\rr d),
\end{aligned}
\end{align}
because the $E$-periodicity of $x\mapsto |V_\phi f(x,\xi )|$
when $f$ is $E$ periodic gives

\begin{equation}\label{Eq:AltFormPerNorms}
\begin{aligned}
g(\xi ) &= \nm {V_\phi f(\cdo ,\xi)}{L^{\mabfr}_E(\kappa (E))}
=
\nm {V_\phi f(\cdo ,\xi)}{L^{\mabfr}_E(x+\kappa (E))},
%\quad \text{and}\quad
\\[1ex]
\nm h{L^{\mabfr}_E(\kappa (E))}
&=
\nm h{L^{\mabfr}_E(x+\kappa (E))},
\qquad \qquad \qquad \qquad \qquad x\in \rr d.
%\text{\hspace{\stretch{1}} $x \in \rr d$.}
\end{aligned}
\end{equation}

\par

\begin{prop}\label{Prop:PerMod}
Let $E$ be a basis of $\rr d$, $r\in(0,1]$,
$\mascB \subseteq L^{r}_{loc}(\rr d)$ be an $E'$-split
Lebesgue space, $\ell _{\mascB ,E}(\Lambda _E)$ be its discrete version,
$\omega _0\in \mascP _E(\rr d)$ and let $\omega (x,\xi )=\omega _0(\xi)$
when $x,\xi \in \rr d$. Then
\begin{equation*}
\maclL ^E(\omega _0,\ell _{\mascB ,E})
=
\splM^{\infty}_{E}(\omega ,\mascB )\bigcap (\maclE _0^E)'(\rr d)
=
\splW^{\infty}_{E}(\omega ,\mascB )\bigcap (\maclE _0^E)'(\rr d).
\end{equation*}
\end{prop}

\par

When proving that $\sfW ^{\mabfr} _{E}(\omega ,\ell ^\mabfp _{E})$
is independent of $\mabfr \in (0,\infty ]^d$ in Section \ref{sec2},
as announced earlier, it will at the same time follow that if
$\mascB$ is a suitable quasi-norm space of Lebesgue type, then
\begin{equation}\label{Eq:ModrumInvariants1}
\begin{alignedat}{3}
\splM ^{\mabfr _1}_E (\omega ,\mascB ) &= \splW ^{\mabfr _2}_E (\omega ,\mascB )
& \quad &\text{when} & \quad \omega &\in \mascP _E(\rr {2d})
%\\[1ex]
%\text{and}\quad
%M^{\mabfr _1,E} (\omega ,\mascB ) &= W^{\mabfr _2,E} (\omega ,\mascB )
%& \quad &\text{when} & \quad \omega &\in \mascP _E(\rr {d})
\end{alignedat}
\end{equation}
for every $\mabfr _1,\mabfr _2\in (0,\infty ]^d$.

\par

\begin{rem}\label{Rem:PerDistr}
The link between periodic Gelfand-Shilov distributions and formal Fourier
series expansions is given by the formula
\begin{equation}\label{Eq:PerDistAction}
\scal f\phi = (2\pi )^{\frac d2}\sum _{\alpha \in \Lambda _E'}c(f,\alpha )
\widehat \phi (-\alpha ).
\end{equation}
\end{rem}

\par

\par

%%%%%%%%%%%%%%%%%%%%%%%%%%%%%%%%%%
\section{Estimates on Wiener spaces and periodic elements
in modulation spaces}
\label{sec2}
%%%%%%%%%%%%%%%%%%%%%%%%%%%%%%%%%%

\par

In this section we deduce equivalences between
various Wiener (quasi-)norm estimates on short-time
Fourier transforms. Especially we prove that \eqref{Eq:ModrumInvariants1}
holds for every $\mabfr _1,\mabfr _2\in (0,\infty ]^d$.

% continuity properties for discrete, semi-discrete
% and non-discrete convolutions. Especially we discuss
% such mapping properties for sequence and Wiener spaces.

\par

%
%We notice that the previous definition makes sense, since 
%$x\mapsto |V_\phi f(x,\xi)|$
%is $E$-periodic when $x\mapsto f(x)$ is $E$-periodic.

\par

\subsection{Estimates of Wiener spaces}

\par

We begin with the following inclusions between
the different Wiener spaces in the previous section.

\par

\begin{prop}\label{Prop:WienerRel1}
Let $(E_1,E_2)$ be permuted dual bases of $\rr d$,
$E=E_1\times E_2$,
$\mabfp ,\mabfq ,\mabfr \in (0,\infty ]^d$ $r_1\in (0,\min (\mabfp ,\mabfq ,\mabfr )]$,
$r_2\in (0,\min (\mabfq )]$, and let $\omega _1,\omega _2
\in \mascP _E(\rr {2d})$ be such that
$$
\omega _1(x,\xi ) =
\omega _2(\xi ,x),\qquad x,\xi \in \rr d.
$$
Then
\begin{multline}\label{Eq:WienerRel1}
\sfW ^{\mabfr ,\infty}_E(\omega ,\ell
^{\mabfp ,\mabfq}_E(\Lambda _E))
\hookrightarrow
\sfW ^{\mabfr}_{1,E_1}(\omega
,\ell _{E_1}^{\mabfp}(\Lambda _{E_1}),L^{\mabfq}_{E_2}(\rr d))
\\[1ex]
\hookrightarrow
\sfW ^{r_1}_E(\omega ,\ell ^{\mabfp ,\mabfq}_E
(\Lambda _E))
\end{multline}
and
\begin{multline}\label{Eq:WienerRel2}
\sfW ^{\infty}_{E'}(\omega ,\ell
^{\mabfq ,\mabfp}_{E'}(\Lambda _{E}'))
\hookrightarrow
\sfW ^{r_2}_{2,E_2'}(\omega ,\ell ^{\mabfp}_{E_2'}(\Lambda _{E_2}'), L^{\mabfq}_{E_1'}(\rr d))
\\[1ex]
\hookrightarrow
\sfW ^{r_2}_{E'}(\omega ,\ell
^{\mabfq ,\mabfp}_{E'}(\Lambda _{E}')).
\end{multline}
%%
%%%
%\begin{alignat}{2}
%\sfW ^{\mabfr ,\infty}_E(\omega _1,\ell
%^{\mabfp ,\mabfq}_E(\Lambda _E))
%&\hookrightarrow &
%\sfW ^{\mabfr}_{1,E_1}(\omega _1,\ell _{E_1}^{\mabfp}(\Lambda _{E_1}),L^{\mabfq}_{E_2}(\rr d))
%&\hookrightarrow
%\sfW ^{r_1}_E(\omega _1,\ell ^{\infty ,\mabfp}_E
%(\Lambda _E))
%\label{Eq:WienerRel1}
%\intertext{and}
%\sfW ^{\infty}_{E'}(\omega _2,\ell
%^{\mabfp ,\mabfq}_{E'}(\Lambda _{E}'))
%&\hookrightarrow &
%\sfW ^{r_2}_{2,E_2}(\omega _2,\ell ^{\mabfp}_{E_1}(\Lambda _{E_1}), L^{\mabfq}(\rr d))
%&\hookrightarrow
%\sfW ^{r_2}_{E'}(\omega _2,\ell
%^{\mabfp ,\mabfq}_{E'}(\Lambda _{E}')).
%\label{Eq:WienerRel2A}
%\end{alignat}
%%%
\end{prop}

%
% r_1=r_0
%
% r_2= gamla q
%
%

\par

\begin{rem}\label{Rem:IncreasingWienerSpaces}
For the involved spaces in Proposition \ref{Prop:WienerRel1} it follows from
H{\"o}lder's inequality that
$$
\sfW ^{\mabfr}_{1,E_1}(\omega ,\ell ^{\mabfp}_{E_1}(\Lambda _{E_k}),L^{\mabfq}_{E_2}(\rr d)),
\quad
\sfW ^{\mabfr}_{2,E_2'}(\omega ,\ell ^{\mabfp}_{E_2'}(\Lambda _{E_2}'),L^{\mabfq}_{E_1'}(\rr d))
$$
and
$$
\sfW ^{\mabfr}_E(\omega ,\ell ^{\mabfp}_{E}(\Lambda _{E}))
$$
increase with respect to $\mabfp$ and decrease with respect to $\mabfr$.
%
%in the previous proposition we observe that
%if $E$ is the same as in Proposition \ref{Prop:WienerRel1},
%$\mabfp , \mabfq _j,\mabfr _j\in (0,\infty ] ^{2d}$, $j=1,2$
%and $\omega \in \mascP _E(\rr {2d})$, then H{\"o}lder's inequality gives
%%%
%\begin{alignat}{2}
%\sfW ^{\mabfq _2}_E(\omega ,\ell ^{\mabfp}_E(\Lambda _E))
%&\subseteq
%\sfW ^{\mabfq _1}_E(\omega ,\ell ^{\mabfp}_E(\Lambda _E)),&
%\qquad \mabfq _1 &\le \mabfq _2,
%\label{Eq:WienerIneq1}
%\\[1ex]
%%
%\intertext{If instead $\mabfp , \mabfq _1,\mabfq _2\in (0,\infty ] ^{d}$
%and $\omega \in \mascP _E(\rr {d})$, then}
%%
%\sfW ^{\mabfq _2}_{1,E_0}(\omega ,L^{\mabfp}_{E_0'}(\rr d))
%&\subseteq
%\sfW ^{\mabfq _1}_{1,E_0}(\omega ,L^{\mabfp}_{E_0'}(\rr d)),&
%\qquad \mabfq _1 &\le \mabfq _2.
%\label{Eq:WienerIneq2}
%\end{alignat}
%%%
%Here, if $\mabfp = (p_1,\dots ,p_d) \in (0,\infty ]^d$
%and $\mabfq = (q_1,\dots ,q_d) \in (0,\infty ]^d$, then
%$\mabfq \le \mabfp$ means that $q_j\le p_j$ for every
%$j\in \{ 1,\dots ,d\}$.
\end{rem}

\par

%\begin{rem}\label{Rem:LebExponents}
%In Proposition \ref{Prop:WienerRel1} and similarly
%in other situations and for other spaces it is understood
%that 
%$$
%\sfW ^{q_0,\mabfr}_E,\quad \sfW ^{\mabfr ,q_0}_E
%\quad \text{and}\quad
%\sfW ^{q_0}_E
%$$
%mean the spaces
%$$
%\sfW ^{\mabfq ,\mabfr}_E,\quad \sfW ^{\mabfr ,\mabfq }_E
%\quad \text{and}\quad
%\sfW ^{\mabfq ,\mabfq}_E
%$$
%when $\mabfq = (q_0,\dots ,q_0)\in (0,\infty ]^d$ and $\mabfr \in (0,\infty ]^d$.
%and $\sfW ^{\mabfr ,\infty}_E$
%mean 
%$\sfW ^{\mabfr ,\mabfr _0}_E$, where
%$\mabfr _0 =(\infty ,\dots ,\infty )\in (0,\infty ]^{d}$.
%In similar situations it is understood that if
%$\mabfr = (\infty ,\dots ,\infty)\in (0,\infty ]^d$,
%$\mabfp _1 \equiv (\mabfr ,\mabfp )
%\in (0,\infty ]^{2d}$
%and
%$\mabfp _2 \equiv (\mabfp ,\mabfr )
%\in (0,\infty ]^{2d}$
%then
%$$
%\ell ^{\infty ,\mabfp}_E(\Lambda _E)
%\equiv 
%\ell ^{\mabfp _1}_E(\Lambda _E)
%\quad \text{and}\quad
%\ell ^{\mabfp ,\infty}_E(\Lambda _E)
%\equiv 
%\ell ^{\mabfp _2}_E(\Lambda _E).
%
% \quad
% (= \ell ^{\mabfp}_{E_0'}(\Lambda _{E_0'};
% \ell ^\infty _{E_0}(\Lambda _{E_0})).
%$$
%\end{rem}

\par

We need the following lemma for the proof of Proposition \ref{Prop:WienerRel1}.
%The right embedding in \eqref{Eq:WienerRel1} follows from the following lemma.

\par

\begin{lemma}\label{Lemma:WienerNormEst2}
Let $\omega \in \mascP _E(\rr d)$, $E$ be an ordered
basis of $\rr d$,
$\kappa (E)$ the parallelepiped spanned
by $E$, $\mabfp \in (0,\infty ]^d$,
$r\in (0,\min (\mabfp )]$ and
let $f$ be measurable on $\rr d$. Then
\begin{equation}\label{Eq:WienerNormEst2}
\nm a{\ell ^{\mabfp}_E(\Lambda _E)}
\lesssim
\nm f{L^{\mabfp}_{E,(\omega )}},
\quad
a(j) = \nm f{L^r_E(j+\kappa (E))}\omega (j).
\end{equation}
\end{lemma}

\par

\begin{proof}
Let $f$ be measurable on $\rr d$, $g_k$ be the same as in
Definition \ref{Def:DiscLebSpaces},
$T_E$ be the linear map which maps the standard basis
into $E$,
$Q_k=[0,1]^k$, and let $\mabfp _k=(p_{k+1},\dots ,p_d)$,
when $k\ge 1$. Then
$$
\nm f{L^r_E(T_E(j)+\kappa (E))}\omega (j)
\asymp
\nm {f\cdot \omega }{L^r_E(T_E(j)+\kappa (E))}
\asymp
\nm {g_0}{L^r_E(j+Q_d)},\quad j\in \zz d.
$$
This reduce the situation to the case that $E$ is the
standard basis, $\kappa (E)=Q_d$ and $\omega =1$.
Moreover, by replacing $|f|^r$ with $f$ and $p_jr$
by $p_j$, $j=1,\dots ,d$, we may assume that $r=1$
(and that each $p_j\ge 1$).

\par

By induction it suffices to prove that if
\begin{align}
a_k(l) &= \nm {g_k}{L^1(l+Q_{d-k})},\quad l\in \zz {d-k},
\notag
\intertext{then}
\nm {a_k}{\ell ^{\mabfp _k}(\zz {d-k})}
&\lesssim
\nm {a_{k+1}}{\ell ^{\mabfp _{k+1}}(\zz {d-k-1})},\quad
k=0,\dots ,d-1,
\label{Eq:WienerNormEst2Ess}
\end{align}
since $\nm {a_0}{\ell ^{\mabfp _0}(\zz {d})}$ is equal to the
left-hand side of \eqref{Eq:WienerNormEst2}, and
$a_d=\nm {a_d}{\ell ^\infty (\zz 0)}$ is equal to
the right-hand side of \eqref{Eq:WienerNormEst2}.

\par

Let $m\in \zz {d-k-1}$ be fixed. We only prove \eqref{Eq:WienerNormEst2Ess}
in the case $p_{k+1}<\infty$. The case $p_{k+1}=\infty$ will follow by similar
arguments and is left for the reader. By first using Minkowski's inequality
and then H{\"o}lder's inequality we get
\begin{multline*}
\nm {a_k(\cdo ,m)}{\ell ^{p_{k+1}}(\mathbf Z)}
=
\left ( \sum _{l_1\in \mathbf Z}\nm {g_k}{L^1((l_1,m)+Q_{d-k})}^{p_{k+1}}
\right )^{\frac 1{p_{k+1}}}
\\[1ex]
=
\left ( \sum _{l_1\in \mathbf Z}
\left (
\int _{m+Q_{d-k-1}}
\left (
\int _{l_1+Q_1}
g_k(t,y)\, dt
\right )\, dy
\right )^{p_{k+1}}
\right )^{\frac 1{p_{k+1}}}
\\[1ex]
\le
\int _{m+Q_{d-k-1}}
\left (
\sum _{l_1\in \mathbf Z}
\left (
\int _{l_1+Q_1}
g_k(t,y)\, dt
\right )^{p_{k+1}}
\right )^{\frac 1{p_{k+1}}}
\, dy
\\[1ex]
\lesssim
\int _{m+Q_{d-k-1}}
\left (
\sum _{l_1\in \mathbf Z}
\int _{l_1+Q_1}
g_k(t,y)^{p_{k+1}}\, dt
\right )^{\frac 1{p_{k+1}}}
\, dy
\\[1ex]
=
\int _{m+Q_{d-k-1}}
\left (
\int _{\mathbf R}
g_k(t,y)^{p_{k+1}}\, dt
\right )^{\frac 1{p_{k+1}}}
\, dy
=
\int _{m+Q_{d-k-1}}
g_{k+1}(y)\, dy.
\end{multline*}
Hence,
\begin{equation}\label{Eq:WienerNormEst2Ess2}
\nm {a_k(\cdo ,m)}{\ell ^{p_{k+1}}(\mathbf Z)}
\lesssim
a_{k+1}(m),\qquad m\in \zz {d-k-1}.
\end{equation}
By applying the $\ell ^{\mabfp _{k+1}}(\zz {d-k-1})$-norm on
\eqref{Eq:WienerNormEst2Ess2} we get
\eqref{Eq:WienerNormEst2Ess}, and thereby
\eqref{Eq:WienerNormEst2}.
\end{proof}

\par

\begin{proof}[Proof of Proposition \ref{Prop:WienerRel1}]
Since the map $F\mapsto F\cdot \omega$ is homeomorphic between
the involved spaces and their corresponding non-weighted versions, we may assume
that $\omega _1=\omega _2=1$. Furthermore, by a linear change of variables,
we may assume that $E_1$ is the standard basis and $E_2=2\pi E_1$. Then
$\kappa (E_1)=Q_d$, $E_1'=E_2$ and $E_2'=E_1$.

\par

Let $F$ be measurable on $\rr {2d}$,
\begin{align*}
f_{1,\mabfr}(\xi ,j) &= \nm {F(\cdo ,\xi )}{L^{\mabfr}(j+\kappa (E_1))},
\quad
g_1(\xi ) = \nm {f_{1,\mabfr}(\xi ,\cdo )}{\ell ^{\mabfp}}
\intertext{and}
G_1(j,\iota ) &= \nm F{L^{\mabfr ,\infty}_{(j,\iota )+\kappa (E_1\times E_2)}}.
\end{align*}
Then
$$
g_1 \le g\equiv
\sum _{\iota +\Lambda _{E_2}}\big ( \nm g{L^\infty (\iota +2\pi Q)}\big ) \cdot
\chi _{\iota +2\pi Q},
$$
and
\begin{equation*} %\label{Eq:FWq1EEst1}
\nm F{\sfW ^{\mabfr}_{1,E_1}(1,\ell ^{\mabfp},L^{\mabfq})}
=
\nm {g_1}{L^{\mabfq}}
\le
\nm g{L^{\mabfq}}
=
\nm {G_1}{\ell ^{\mabfp ,\mabfq}}
\asymp
\nm F{\sfW ^{\mabfr ,\infty}_E (1,\ell ^{\mabfp ,\mabfq})}.
\end{equation*}
This implies that
$\sfW ^{\mabfr ,\infty}_E (1,\ell ^{\mabfp ,\mabfq})
\hookrightarrow
\sfW ^{\mabfr}_{1,E_1}(1,\ell ^{\mabfp},L^{\mabfq})$,
and the first inclusion in \eqref{Eq:WienerRel1} follows.
%where $a(\iota ) = \nm f{L^\infty (\iota +\kappa (E_0'))}$.

\par

In order to prove the second inclusion in
\eqref{Eq:WienerRel1}, we may assume
that $r_0<\infty$, since otherwise the result is trivial. Let
\begin{align*}
%\fy (\xi ,j) &= \nm {F(\cdo ,\xi )}{L^{r_0}(j+E_1)},&
\quad %\text{and}\quad
\psi (\xi ) &= \nm {f_{1,r_0} (\xi ,\cdo )}{\ell ^{\mabfp}}, 
\quad
a(\iota ) = \nm \psi {L^{\mabfq}(\iota +\kappa (E_2))}
\intertext{and}
H_1(j,\iota ) &= \nm {f_{1,r_0} (\cdo ,j)}{L^{r_0}(\iota 
+\kappa (E_2))}.
\end{align*}
Then
$$
\nm \psi{L^{\mabfq} (\rr d)} 
=
\nm F{\sfW _{1,E_1}^{r_0}(1,\ell ^{\mabfp} ,L^{\mabfq})}
\quad \text{and}\quad
\nm {H_1}{\ell ^{\mabfp ,\mabfq}}
=
\nm F{\sfW _{E}^{r_0}(1,\ell ^{\mabfp ,\mabfq})}.
$$

\par

By Minkowski's inequality and the fact that
$\min (\mabfp )\ge r_0$ we get
\begin{multline*}
\nm {H_1(\cdo ,\iota )}{\ell ^{\mabfp}}
%\\[1ex]
=
\Nm {\left ( \int _{\iota +\kappa (E_2)} f_{1,r_0}
(\xi ,\cdo )^{r_0}\, d\xi
\right )^{\frac 1{r_0}}}{\ell ^{\mabfp}}
\\[1ex]
=
\left (\Nm {\left ( \int _{\iota +\kappa (E_2)}
f_{1,r_0} (\xi ,\cdo )^{r_0}\, d\xi
\right )}{\ell ^{\mabfp /{r_0}}}\right ) ^{\frac 1{r_0}}
\\[1ex]
\le
\left (\int _{\iota +\kappa (E_2)}
\nm {f_{1,r_0} (\xi ,\cdo )^{r_0}}{\ell ^{\mabfp /{r_0}}}
\, d\xi \right ) ^{\frac 1{r_0}}
= a(\iota ).
%\nm \psi{L^{r_0}(\iota +\kappa (E_2))}.
\end{multline*}
Hence $\nm {H_1}{\ell ^{\mabfp ,\mabfq}}
\le \nm a{\ell ^{\mabfq}}$.
By Lemma \ref{Lemma:WienerNormEst2} it follows that
$\nm a{\ell ^{\mabfq}}\le \nm \psi{L^{\mabfq}}$, and the 
second inclusion
of \eqref{Eq:WienerRel1} follows by combining these relations.

\par

It remains to prove \eqref{Eq:WienerRel2}. Again we may assume that
$r_2<\infty$, since otherwise the result is trivial. Let
\begin{align*}
f_{2,q}(x,\iota ) &= \nm {F(x,\cdo )}{L^{q}(\iota +\kappa (E_2))},
\quad
f_3(x) = \nm {F(x,\cdo )}{L^{\mabfq}(\rr d)},
\\[1ex]
H_{2,q_1,q_2}(\iota ,j) &= \nm {f_{2,q_1}(\cdo ,\iota )}{L^{q_2} (j+\kappa (E_1))},
\quad \text{and}\quad H_{2,q}=H_{2,q,q}
\end{align*}
when $q,q_1,q_2\in (0,\infty ]$. Then the fact that $r_2\le \min (\mabfq)$, Minkowski's inequality and
Lemma \ref{Lemma:WienerNormEst2} give
\begin{multline*}
\nm {H_{2,r_2}(\cdo ,j)}{\ell ^{\mabfq}}
\le 
\left (
\int _{j+\kappa (E_1)}
\nm {f_{2,r_2}(x,\cdo )}{\ell ^{\mabfq}}^{r_2}\, dx
\right )^{\frac 1{r_2}}
\\[1ex]
\le 
\left (
\int _{j+\kappa (E_1)}
\nm {F(x,\cdo )}{L^{\mabfq}}^{r_2}\, dx
\right )^{\frac 1{r_2}}.
\end{multline*}
By applying the $\ell ^{\mabfp}$ norm on the latter inequality we get
$$
\nm F{\sfW ^{r_2}_{E'}(1,\ell ^{\mabfq ,\mabfp})}
\le
\nm F{\sfW ^{r_2}_{2,E_2}(1,\ell ^{\mabfp},L^{\mabfq})},
$$
and the second relation in \eqref{Eq:WienerRel2} follows.

\par

On the other hand,  we have
\begin{multline*}
\left (
\int _{j+\kappa (E_1)}
\nm {F(x,\cdo )}{L^{\mabfq}(\rr d)}^{r_2}
\right )^{\frac 1{r_2}}\, dx
\lesssim
\left (
\int _{j+\kappa (E_1)}
\nm {f_{2,\infty}(x,\cdo )}{\ell ^{\mabfq}}^{r_2}\, dx
\right )^{\frac 1{r_2}}
\\[1ex]
\le \nm {H_{2,\infty}(\cdo ,j)}{\ell ^{\mabfq}}
\end{multline*}
Again, by applying the $\ell ^{\mabfp}$ norm with respect to the $j$ variable,
we get
$$
\nm F{\sfW ^{r_2}_{2,E_2}(1,\ell ^{\mabfp},L^{\mabfq})}
\le
\nm F{\sfW ^{\infty}_{E'}(1,\ell ^{\mabfq ,\mabfp})},
$$
and the first relation in \eqref{Eq:WienerRel2} follows.
\end{proof}

\par

\subsection{Wiener estimates on short-time Fourier
transforms, and modulation spaces}

\par

Essential parts of our analysis
are based on Lebesgue estimates of the \emph{semi-discrete
convolution} with respect to (the ordered) basis $E$ in $\rr d$, given by
\begin{equation}\label{EqDistSemContConv}
(a*_{[E]}f)(x) \sum _{j\in \Lambda _E}a(j)f(x - j),
\end{equation}
when $f\in \maclS _{1/2}'(\rr d)$ and $a \in \ell _0 (\Lambda _E)$.

\par

The next result is an extension of
\cite[Proposition 2.1]{Toft15} and
\cite[Lemma 2.6]{GaSa}, but a special case of
\cite[Theorem 2.1]{Toft22}.
The proof is therefore omitted.
Here the domain of integration is of the form
\begin{equation}\label{Eq:IDomain1}
I=\sets {x_1e_1+\cdots +x_de_d}{x_k\in J_k},
\quad
J_k =
\begin{cases}
[0,1], & e_k\in E_0
\\[1ex]
\mathbf R, & e_k\notin E_0
\end{cases}
\end{equation}

\par

\begin{prop}\label{Prop:SemiContConvEstNonPer}
Let $E$ be an ordered basis
of $\rr d$, $E_0\subseteq E$, $I$ be given by
\eqref{Eq:IDomain1},
$\omega ,v\in \mascP _E(\rr d)$ be
such that $\omega$ is $v$-moderate, and let
$\mabfp ,\mabfr \in (0,\infty ]^{d}$ be such that
$$
r_k\le \min _{m\le k}(1,p_m).
$$
Also let  $f$ be measurable on $\rr d$ such that $|f|$ is $E_0$-periodic
and $f\in L^{\mabfp}_{E,(\omega )}(I)$.
Then the map $a\mapsto a*_{[E]}f$ from $\ell _0(\Lambda _E)$ to
$L^{\mabfp}_{E,(\omega )}(I)$ extends uniquely to a
linear and continuous map from $\ell ^{\mabfr}_{E ,(v)}(\Lambda _E)$ to
$L^{\mabfp}_{E,(\omega )}(I)$, and
\begin{equation}\label{convest1}
\nm {a*_{[E]}f}{L^{\mabfp }_{E,(\omega )}(I)}\le
C\nm {a}{\ell ^{\mabfr}_{E ,(v)}(\Lambda _E)}
\nm {f}{L^{\mabfp }_{E,(\omega )}(I)},
\end{equation}
for some constant $C>0$ which is independent of
$a\in \ell ^{\mabfr}_{E ,(v)}(\Lambda _E)$ and
measurable $f$ on $\rr d$ such that $|f|$ is $E_0$-periodic.
\end{prop}

\par

We have now the following result, which agrees with
Proposition \ref{Prop:WienerEquiv} when
$\mabfr =( \infty ,\dots ,\infty )$.

\par

\renewcommand{\rubrik}{Proposition \ref{Prop:WienerEquiv}$'$}

\par

\begin{tom}
Let $E$ be a phase split basis for $\rr {2d}$, $\mabfp ,\mabfr \in (0,\infty ]^{2d}$,
$r\in (0,\min (1,\mabfp )]$, $\omega ,v\in \mascP _E(\rr {2d})$
be such that $\omega$ is $v$-moderate, $\rho$ and $\Theta _\rho v$ $\rho$ be
as in \eqref{Eq:SubMultMod} with strict inequality when $r<1$, and let $\phi _1,\phi _2\in
M^1_{(\Theta _\rho v)} (\rr d)\setminus 0$.
%$\sigma \in \operatorname {S}_{2d}$
Then
$$
\nm f{M^{\mabfp}_{E,(\omega )}}
\asymp
\nm {V_{\phi _1} f}{L^\mabfp _{E,(\omega )}}
\asymp 
\nm {V_{\phi _2} f}{\sfW ^{\mabfr} _{E}(\omega ,\ell ^\mabfp _{E})},\quad f\in
\maclS '_{1/2}(\rr d).
$$
In particular, if $f\in \maclS _{1/2}'(\rr d)$, then
$$
f\in M^{\mabfp}_{E,(\omega )}(\rr {2d})
\quad \Leftrightarrow \quad
V_{\phi _1} f \in L^\mabfp _{E,(\omega )}(\rr {2d})
\quad \Leftrightarrow \quad
V_{\phi _2} f \in \sfW ^{\mabfr} _{E}(\omega ,\ell ^\mabfp _{E}
(\Lambda _E)).
$$
\end{tom}

\par

We need the following lemma for the proof.

\par

\begin{lemma}\label{SubharmonicEst}
Let $p\in (0,\infty ]$, $r>0$, $(x_0,\xi _0)\in \rr {2d}$ be fixed,
and let $\phi \in \maclS _{1/2}(\rr d)$ be a Gaussian. Then
$$
|V_\phi f(x_0,\xi _0)| \le C\nm {V_\phi f}{L^p(B_r(x_0,\xi _0))},\quad
f\in \maclS _{1/2}'(\rr d),
$$
where the constant $C$ is independent of $(x_0,\xi _0)$ and $f$.
\end{lemma}

\par

When proving Lemma \ref{SubharmonicEst} we may first reduce ourself
to the case that the Gaussian $\phi$ should be centered at origin, by
straight-forward arguments involving pullbacks with translations.
The result then follows by using the same arguments as in
\cite[Lemma 2.3]{GaSa} and its proof, based on the fact that
$$
z\mapsto F_w(z) \equiv e^{c_1|z|^2+c_2(z,w)+c_3|w|^3}V_\phi f(x,\xi ),
\quad z=x+i\xi
$$
is an entire function for some choice of the constant
$c_1$ (depending on $\phi$).

\par

\begin{proof}[Proof of Proposition
\ref{Prop:WienerEquiv}${}^{\, \prime}$]
Let
$F=V_{\phi} f$, $F_0=V_{\phi _0} f$, $\kappa (E)$
be the (closed) parallelepiped which is spanned by
$E=\{ e_1,\dots ,e_{2d}\}$, and let
$$
\kappa _M(E) = \sets {x_1e_1+\cdots +x_{2d}e_{2d}}{|x_k|\le 2,\ k=1,\dots ,2d}.
$$
Also choose $r_0>0$ small enough such that
$$
\kappa (E)+B_{r_0}(0,0) \subseteq \kappa _M(E)
$$
The result holds when $\mabfr =(\infty ,\dots ,\infty )$, in view of Proposition
\ref{Prop:WienerEquiv}. By H{\"o}lder's inequality we also have
\begin{equation}\label{Eq:WienerSTFT1}
\nm {V_{\phi}f}{\sfW ^{\mabfr} _{E}(\omega ,\ell ^\mabfp _{\sigma})}
\lesssim
\nm {V_{\phi}f}{\sfW ^\infty _{E}(\omega ,\ell ^\mabfp _{\sigma})}.
\end{equation}
We need to prove the reversed inequality
\begin{equation}\label{Eq:WienerSTFT2}
\nm {V_{\phi}f}{\sfW ^\infty _{E}(\omega ,\ell ^\mabfp _{\sigma})}
\lesssim
\nm {V_{\phi}f}{\sfW ^{\mabfr} _{E}(\omega ,\ell ^\mabfp _{\sigma})},
\end{equation}
and it suffices to prove this for $\mabfr = (r,\dots ,r)$ for some $r\in (0,1]$
in view of H{\"o}lder's inquality.

\par

First we consider the case when $\phi =\phi _0$. If $r>0$ is small enough
and $j\in \Lambda _E$, then
Lemma \ref{SubharmonicEst} gives
for some $(x_j,\xi _j)\in j+\kappa (E)$ that
\begin{equation*}
\nm {V_{\phi _0}f}{L^\infty (j+\kappa (E))}
=
|{V_{\phi _0}f}(x_j,\xi _j)|
\lesssim
\nm {V_{\phi _0}f}{L^r (B_r(x_j,\xi _j))}
\le
\nm {V_{\phi _0}f}{L^r (j+\kappa _M(E))}
\end{equation*}
Hence,
\begin{multline*}
\nm {V_{\phi _0}f}{\sfW ^\infty _{E}(\omega ,\ell ^\mabfp _{E})}
=
\nm { \{ \nm {V_{\phi _0}f}{L^\infty (j+\kappa (E)}\omega (j)\}
_{j\in \Lambda _E}}{\ell ^\mabfp _{E}(\Lambda _E)}
\\[1ex]
\lesssim
\nm { \{ \nm {V_{\phi _0}f}{L^r (j+\kappa _M(E)}\omega (j)\}
_{j\in \Lambda _E}}{\ell ^\mabfp _{E}(\Lambda _E)}
\\[1ex]
\asymp
\nm { \{ \nm {V_{\phi _0}f}{L^r (j+\kappa (E)}\omega (j)\}
_{j\in \Lambda _E}}{\ell ^\mabfp _{E}(\Lambda _E)}
=
\nm {V_{\phi _0}f}{\sfW ^r _{E}(\omega ,\ell ^\mabfp _{E})},
\end{multline*}
and \eqref{Eq:WienerSTFT2} holds for $\phi = \phi _0$.

\par

Next suppose that $\phi$ is arbitrary and let $n\ge 1$ be
a large enough integer such that if
$$
E_n = \frac 1n \cdot E \equiv \left \{  \frac {e_1}n,\dots ,\frac {e_{2d}}n \right \}
\quad \text{and}\quad
\Lambda
= \frac 1n \Lambda _E = \Lambda _{E_n},
$$
then
$$
\{ \phi (\cdo -k)e^{i\scal \cdo \kappa}\} _{(k,\kappa )\in \Lambda}
$$
is a frame. Since $\phi \in M^1_{(\Theta _\rho v)}$, it follows
that its canonical dual $\psi$ also belongs to
$M^1_{(\Theta _\rho v)}$ (cf. \cite[Theorem S]{Gc1}).
Consequently, any $f$ possess the expansions
\begin{align}
f
&=
\sum _{(k,\kappa )\in \Lambda} V_\phi f(k,\kappa )
\psi (\cdo -k)e^{i\scal \cdo \kappa}
\notag
\\[1ex]
&=
\sum _{(k,\kappa )\in \Lambda} V_\psi f(k,\kappa )
\phi (\cdo -k)e^{i\scal \cdo \kappa}
\label{Eq:GaborExp}
\end{align}
with suitable interpretation of convergences.

\par

Let
$$
F_0=|V_{\phi _0}f|\cdot \omega ,
\quad
F=|V_{\phi }f|\cdot \omega ,
\quad \text{and}\quad
a(\mabfk ) = |V_\psi \phi _0 (-\mabfk )| .
$$
As in the proofs of \cite[Theorem 3.1]{GaSa} and
\cite[Proposition 3.1]{Toft15} we use the fact that
\begin{equation}\label{Eq:STFTWindTrans}
|V_{\phi _0}f| \le (2\pi )^{-\frac d2}
a *_{[E_n ]}|V_\phi f| ,
\end{equation}
which follows from
\begin{multline*}
|V_{\phi _0}f(x,\xi )| = (2\pi )^{-\frac d2}|(f,e^{i\scal \cdo {\xi}}
\phi _0(\cdo -x))|
\\[1ex]
\le (2\pi )^{-\frac d2}\sum _{(k,\kappa )\in \Lambda}
|(V_\psi \phi _0)(k,\kappa )|  |(f,e^{i\scal \cdo {\xi +\kappa}}
\phi (\cdo -x-k))|
\\[1ex]
=  (2\pi )^{-\frac d2}\sum _{(k,\kappa )\in \Lambda}
|(V_\psi \phi _0)(k,\kappa )|  |V_{\phi }f(x+k,\xi +\kappa )|
\\[1ex]
=
(a *_{[E_n ]} |V_{\phi }f|)(x,\xi  ).
\end{multline*}
Here we have used \eqref{Eq:GaborExp} with $\phi _0$ in place of $f$,
in the inequality. By using that
$$
\omega (x,\xi )\lesssim v(k,\kappa )\omega (x+k,\xi +\kappa ),
$$
\eqref{Eq:STFTWindTrans} gives
\begin{equation}\label{Eq:F0FConvEst}
F_0 \lesssim (a\cdot v)*_{[E_n]}F.
\end{equation}
If we set
$$
b_0(j) = \int _{j+\kappa (E)} |F_0(X)|^r\, dX
\quad \text{and}\quad
b(j) = \int _{j+\kappa (E)} |F(X)|^r\, dX,
\qquad j\in \Lambda ,
$$
integrate \eqref{Eq:F0FConvEst} and use the fact
that $r\le 1$, we get for $j\in \Lambda$ that
\begin{multline*}
b_0(j)
\lesssim
\int _{j+\kappa (E)}
\left (
\sum _{k\in \Lambda}a(k)v(k)|F(X-k)|
\right )^r
\, dX
\\[1ex]
\lesssim
\sum _{k\in \Lambda}(a(k)v(k))^r \int _{j+\kappa (E)}
|F(X-k)|^r\, dX 
= ((a\cdot v)^r * b ))(j),
\end{multline*}
where $*$ is the discrete convolution with respect to the
lattice $\Lambda$.

\par

Let $\mabfq = \mabfp /r$.
Then $\min (\mabfq )\ge 1$,
and Young's inequality applied on the last inequality
gives
\begin{multline}\label{Eq:F0FComputations}
\nm {F_0}{\sfW ^r_E(1,\ell ^{\mabfp}_E )}
=
\nm {b_0^{\frac 1r}}{\ell ^{\mabfp} _E (\Lambda _E)}
\lesssim
\nm {(a\cdot v)^r*b}{\ell ^{\mabfq} _E (\Lambda _E)}
^{\frac 1r}
\\[1ex]
\le
\nm {(a\cdot v)^r*b}{\ell ^{\mabfq} _E (\Lambda )}^{\frac 1r}
\le
\left (
\nm {(a\cdot v)^r}{\ell ^1(\Lambda )}
\nm b{\ell ^{\mabfq} _E (\Lambda )}
\right )^{\frac 1r}
\\[1ex]
\asymp
\nm {a}{\ell ^r_{(v)}(\Lambda )}
\nm b{\ell ^{\mabfq} _E (\Lambda )}
^{\frac 1r}
\le
\nm {a}{\ell ^1_{(\Theta _\rho v)}(\Lambda )}
\nm {b^{\frac 1r}}{\ell ^{\mabfp} _E (\Lambda )}
\\[1ex]
\lesssim
\nm {\phi}{M^1_{(\Theta _\rho v)}}
\nm {b^{\frac 1r}}{\ell ^{\mabfp} _E (\Lambda )}.
\end{multline}
In the last steps we have used H{\"o}lder's inequality and
$$
\nm {V_{\phi _0}\phi}{L^1_{(\Theta _\rho v)}(\rr {2d})}
\asymp
\nm { \{ \nm {V_{\phi _0}\phi}{L^\infty (j+\kappa (E)}
(\Theta _\rho v)(j)\}
_{j\in \Lambda _E}}{\ell ^1(\Lambda _E)}
\asymp \nm \phi {M^1_{(\Theta _\rho v)}}.
$$

\par

We have
\begin{equation*}
\nm {b^{\frac 1r}}{\ell ^{\mabfp} _E (\Lambda )}
\\[1ex]
=
\nm { \{ \nm F{L^r(j+\kappa (E))} \} _{j\in \Lambda }}
{\ell ^{\mabfp}_E (\Lambda )},
\end{equation*}
$\bigcup _{j\in \Lambda _E}(j+\kappa (E))
=\rr d$, $\Lambda _E$ and $\Lambda$ are lattices
such that $\Lambda$ contains $\Lambda _E$, and
$\Lambda$ is $n$ times as dense as $\Lambda _E$. From
these facts it follows by straight-forward computations
that
\begin{multline*}
\nm { \{ \nm F{L^r(j+\kappa (E))} \} _{j\in \Lambda }}
{\ell ^{\mabfp}_E (\Lambda )}
\asymp
\nm { \{ \nm F{L^r(j+\kappa (E))} \} _{j\in \Lambda _E}}
{\ell ^{\mabfp}_E (\Lambda _E)}
\\[1ex]
\asymp
\nm { \{ \nm {V_\phi f}{L^r(j+\kappa (E))}\omega (j)
\} _{j\in \Lambda _E}}
{\ell ^{\mabfp}_E (\Lambda _E)}
= \nm F{\sfW ^r_E(\omega ,\ell ^{\mabfp}_E )}.
\end{multline*}
Here the second relation follows from the fact that
$\omega (x)\asymp \omega (j)$ when $j\in \Lambda _E$ and
$x\in j+\kappa (E)$, which follows from \eqref{moderate}.
By combining these relations with \eqref{Eq:F0FComputations}
we get
$$
\nm {F_0}{\sfW ^r_E(1,\ell ^{\mabfp}_E )}
\lesssim
\nm {F}{\sfW ^r_E(1,\ell ^{\mabfp}_E )}.
$$
Hence, Proposition \ref{Prop:WienerEquiv} and the fact that we have
already proved \eqref{Eq:WienerSTFT2} when $\phi$ equals $\phi _0$
gives
\begin{multline*}
\nm {V_{\phi}f}{\sfW ^\infty _{E}(\omega ,\ell ^\mabfp _{\sigma})}
\asymp
\nm {V_{\phi _0}f}{\sfW ^\infty _{E}(\omega ,\ell ^\mabfp _{\sigma})}
\asymp
\nm {F_0}{\sfW ^\infty _{E}(1,\ell ^\mabfp _{\sigma})}
\lesssim
\nm {F_0}{\sfW ^{\mabfr} _{E}(1,\ell ^\mabfp _{\sigma})}
\\[1ex]
\lesssim
\nm {F}{\sfW ^{\mabfr} _{E}(1,\ell ^\mabfp _{\sigma})}
\asymp
\nm {V_{\phi}f}{\sfW ^{\mabfr} _{E}(\omega ,\ell ^\mabfp _{\sigma})}.
\qedhere
\end{multline*}
%%
%The result now follows by combining these relations with
%\eqref{Eq:F0FComputations}.
\end{proof}

\par

By combining Proposition \ref{Prop:WienerEquiv}$'$ with
Proposition \ref{Prop:WienerRel1}
and Remark \ref{Rem:IncreasingWienerSpaces}
we get the following.

\par

\begin{prop}\label{Prop:SpecCaseWienerRel1}
Let $E_0$ be a basis for $\rr d$, $E_0'$ be its dual basis,
$E=E_0\times E_0'$, $\mabfq ,\mabfr \in (0,\infty ]^{d}$,
$\omega _0,v_0\in \mascP _E(\rr {d})$
be such that $\omega _0$ is $v_0$-moderate,
$\omega (x,\xi )=\omega _0(\xi )$,
$v(x,\xi )=v_0(\xi )$, $\Theta _\rho v$ be
as in \eqref{Eq:SubMultMod} with strict inequality when $r<1$, and let $\phi \in
M^1_{(\Theta _\rho v)} (\rr d)\setminus 0$.
%$\sigma \in \operatorname {S}_{2d}$
Then
\begin{alignat*}{2}
M^{\infty ,\mabfq}_{E,(\omega )}(\rr d)
&=
\splM ^{\mabfr}_{E_0}(\omega _0,L^{\mabfq} _{E_0'}(\rr d)),&
\quad
W^{\infty ,\mabfq}_{E,(\omega )}(\rr d)
&=
W^{\mabfr}_{E_0}(\omega _0,L^{\mabfq} _{E_0'}(\rr d)),
\intertext{and}
\nm f{M^{\infty ,\mabfq}_{E,(\omega )}}
&\asymp
\nm {V_\phi f}
{\sfW ^{\mabfr}_{1,E_0}(\omega ,\ell ^\infty _E,L^{\mabfq}_{E_0'})},&
\quad
\nm f{W^{\infty ,\mabfq}_{E,(\omega )}}
&\asymp
\nm {V_\phi f}
{\sfW ^{\mabfr}_{2,E_0}(\omega ,\ell ^\infty _E,L^{\mabfq}_{E_0'})}.
\end{alignat*}
\end{prop}

\par

\subsection{Periodic elements in modulation spaces}

\par

By a straight-forward combination of
Propositions \ref{Prop:PerMod}  and
\ref{Prop:SpecCaseWienerRel1}
we get the following. The details are left for the reader.

\par

\begin{prop}\label{Prop:PeriodicMod}
Let $E_0$ be a basis for $\rr d$, $E_0'$ be its dual basis,
$E=E_0\times E_0'$, $\mabfq ,\mabfr \in (0,\infty ]^{d}$,
$\omega _0,v_0\in \mascP _E(\rr {d})$
be such that $\omega _0$ is $v_0$-moderate,
$\omega (x,\xi )=\omega _0(\xi )$,
$v(x,\xi )=v_0(\xi )$, $\Theta _\rho v$ be
as in \eqref{Eq:SubMultMod} with strict inequality
when $r<1$, and let $\phi \in
M^1_{(\Theta _\rho v)} (\rr d)\setminus 0$.
%$\sigma \in \operatorname {S}_{2d}$
Then
\begin{multline}\label{Eq:NormsEquivPerElem}
\nm f{M^{\infty ,\mabfq}_{E,(\omega )}}
\asymp 
\nm f{W^{\infty ,\mabfq}_{E,(\omega )}}
\asymp
\nm f{\splM ^{\mabfr}_{E_0}(\omega ,L^{\mabfq} _{E_0'})}
\\[1ex]
\asymp
\nm f{\sfW ^{\mabfr}_{E_0}(\omega ,L^{\mabfq} _{E_0'})}
\asymp
\nm f{\maclL ^E(\omega _0,\ell ^{\mabfq}_{E_0'}(\Lambda _{E_0}'))},
\qquad f\in (\maclE _0^{E})'(\rr d).
\end{multline}
\end{prop}

\par

As an immediate consequence of the previous result we get the
following extension of Proposition \ref{Prop:PerMod}. The details
are left for the reader.

\par

\renewcommand{\rubrik}{Proposition \ref{Prop:PerMod}$'$}

\par

\begin{tom}
Let $E$ be a basis of $\rr d$, $\mabfr \in (0,\infty ]^d$, $r\in(0,1]$,
$\mascB \subseteq L^{r}_{loc}(\rr d)$ be an $E'$-split
Lebesgue space, $\ell _{\mascB ,E}(\Lambda _E)$ its discrete version,
and let $\omega\in \mascP _E(\rr d)$. Then
\begin{equation*}
\maclL ^{E}(\omega ,\ell _{\mascB ,E})
=
\splM^{\mabfr}_{E}(\omega ,\mascB )\bigcap (\maclE _0^E)'(\rr d)
=
\splW^{\mabfr}_{E}(\omega ,\mascB )\bigcap (\maclE _0^E)'(\rr d)
\end{equation*}
\end{tom}

\par

\begin{rem}
Let
$$
E_0= \{ e_1,\dots ,e_d\} ,\quad 
E_0'= \{ \ep _1,\dots ,\ep _d\} ,\quad 
\mabfq = (q_1,\dots ,q_d),\quad
\mabfr = (r_1,\dots ,r_d),
$$
$\omega$, $v$  and $\phi$ be the same as in Proposition
\ref{Prop:PeriodicMod}, %$(T\omega )(\xi ,x) = \omega _0(\xi )$,
and let $r_0\le \min (\mabfr )$ and $f\in f\in (\maclE _0^{E})'(\rr d)$
with Fourier series expansion \eqref{Eq:Expan2}.
Then \eqref{Eq:ModPerNormEquiv1}--\eqref{Eq:AltFormPerNorms} and
\eqref{Eq:NormsEquivPerElem} imply that 
\begin{equation}\label{Eq:ModPerNormEquiv3}
\nm {V_\phi f \cdot \omega}{L^{\mabfr ,\mabfq}_E(\kappa (E_0)\times \rr d)}
\asymp
\nm {V_\phi f \cdot \omega}{L^{\mabfq ,\mabfr}_{E'}(\rr d \times \kappa (E_0))}
\asymp
\nm {c(f,\cdo )}{\ell ^{\mabfq}_{E_0',(\omega _0)}}.
\end{equation}

\par

Let $\nmm \cdo$ be the quasi-norm on the left-hand side of
\eqref{Eq:ModPerNormEquiv3}, after the orders of the
involved $L^{q_k}_{e_k'}(\mathbf R)$ and $L^{r_k}_{e_k}(\kappa (e_k))$
quasi-norms have been permuted in such way that the internal order of the hitting
$L^{q_k}_{e_k'}(\mathbf R)$ quasi-norms remains the same. Then
\begin{equation}\label{Eq:PermutNormEst}
\nm F{L^{r_0,\mabfq}_E(\kappa (E_0)\times \rr d)}
\lesssim
\nmm F
\lesssim
\nm F{L^{\infty,\mabfq}_E(\kappa (E_0)\times \rr d)},
\end{equation}
by repeated application of H{\"o}lder's inequality. A combination of 
\eqref{Eq:ModPerNormEquiv3} and \eqref{Eq:PermutNormEst} give
\begin{equation}\label{Eq:ModPerNormEquiv4}
\nmm {V_\phi f \cdot \omega}
\asymp
\nm {c(f,\cdo )}{\ell ^{\mabfq}_{E_0',(\omega _0)}}.
\end{equation}

\par

In particular, if $e_j$ are the same as in Remark
\ref{Rem:CommentPhasSplit}, $E_*$ is the ordered
basis $\{ e_1,e_{d+1}, \dots ,e_d,e_{2d}\}$ of $\rr {2d}$,
$$
\Omega
=
\sets {y_1e_1+\cdots +y_{2d}e_{2d}}
{0\le y_j\le 1\ \text{and}\ y_{d+j}\in \mathbf R,\, j=1,\dots ,d}
$$
and $\mabfq _0=(q_1,q_1,q_2,q_2,\dots ,q_d,q_d)\in (0,\infty ]^{2d}$, then
\begin{equation}\label{Eq:ModPerNormEquiv5}
\nm {V_\phi f \cdot \omega}{L^{\mabfq _0}_{E_*}(\Omega )}
\asymp
\nm {c(f,\cdo )}{\ell ^{\mabfq}_{E_0',(\omega _0)}}.
\end{equation}
\end{rem}

\par

\begin{rem}
With the same notation as in the previous remark, we note that
if $E_0'$ is the standar basis of $\rr d$, $X_j=(x_j,\xi _j)$, $j=1,\dots ,d$,
$I=\mathbf R\times [0,2\pi]$ and $\max (\mabfq )<\infty$,
then \eqref{Eq:ModPerNormEquiv5} is the same as
\begin{multline}\tag*{(\ref{Eq:ModPerNormEquiv5})$'$}
\left (
\int _{I}
\left (
\cdots
\left (
\int _{I}
|V_\phi f(x,\xi )\omega _0(\xi )|^{q_1}\, dX_1
\right )^{\frac {q_2}{q_1}}
\cdots
\right )^{\frac {p_d}{p_{d-1}}}
\, dX_d\right )^{\frac 1{p_d}}
\\[1ex]
\asymp
\left (
\sum _{\alpha _d\in \mathbf Z}
\left (
\cdots
\left (
\sum _{\alpha _1\in \mathbf Z}
|c(f,\alpha )\omega _0(\alpha )|^{q_1}
\right )^{\frac {q_2}{q_1}}
\cdots
\right )^{\frac {p_d}{p_{d-1}}}
\right )^{\frac 1{p_d}}
\end{multline}
\end{rem}

\par

\end{document}